\documentclass[review,3p,square]{elsarticle}
\usepackage{pdfsync}
\usepackage{amssymb}
\usepackage{amsmath}
\usepackage{amsthm}
\usepackage{stmaryrd}
\usepackage{natbib}
\usepackage{bbm}
\usepackage{amsmath}
\allowdisplaybreaks[4]
\usepackage{enumitem}
\usepackage{fancyhdr}
\usepackage{hyperref}
\hypersetup{%
  pdftitle={BSDEs},%
  pdfsubject={BSDEs},%
  pdfauthor={YaQi Zhang, ShengJun FAN},%
  pdfkeywords={BSDEs},%
  pdfstartview=FitH,%
  CJKbookmarks=true,%
  bookmarksnumbered=true,%
  bookmarksopen=true,%
  colorlinks=true, linkcolor=blue, urlcolor=blue, citecolor=blue, %
}

\usepackage{cleveref}
\crefname{thm}{Theorem}{Theorems}
\crefname{pro}{Proposition}{Propositions}
\crefname{lem}{Lemma}{Lemmas}
\crefname{rmk}{Remark}{Remarks}
\crefname{cor}{Corollary}{Corollaries}
\crefname{dfn}{Definition}{Definitions}
\crefname{ex}{Example}{Examples}
\crefname{section}{Section}{Sections}
\crefname{subsection}{Subsection}{Subsections}

\newcommand{\To}{\rightarrow}
\newcommand{\as}{{\rm d}\mathbb{P}\times{\rm d} t-a.e.}
\newcommand{\ps}{\mathbb{P}-a.s.}

\newcommand{\essinf}{\mathop{\operatorname{ess\,inf}}}
\newcommand{\F}{\mathcal{F}}
\newcommand{\E}{\mathbb{E}}

\newcommand{\R}{{\mathbb R}}

\newcommand{\RE}{\forall}

\newtheorem{thm}{Theorem}[section]

\newtheorem{pro}[thm]{Proposition}
\newtheorem{rmk}[thm]{Remark}
\newtheorem{cor}[thm]{Corollary}

\newtheorem{ex}[thm]{Example}

\DeclareMathOperator{\sgn}{sgn}

\journal{arXiv}

\begin{document}
\begin{frontmatter}

\title{{Weighted $L^p$ solutions of scalar BSDEs with general unbounded stochastic coefficients}\tnoteref{found}}
\tnotetext[found]{Partially supported by National Natural Science Foundation of China (No. 12171471).
\vspace{0.2cm}}

\author{Yaqi Zhang} \ead{TS22080028A31@cumt.edu.cn}
\author{Zongjia Zhu} \ead{1284874259@qq.com}
\author{Shengjun Fan\corref{cor}} \ead{shengjunfan@cumt.edu.cn}\vspace{-0.6cm}

\address{School of Mathematics, China University of Mining and Technology, Xuzhou 221116, China}

\cortext[cor]{Corresponding \vspace{0.2cm}author}
\vspace{0.2cm}

\begin{abstract}
This paper is devoted to solving one-dimensional backward stochastic differential equations (BSDEs in short) with a general random terminal time $\tau$ taking values in the extended nonnegative real numbers. The generator $g$ of BSDEs satisfies some stochastic growth/continuity conditions in the state variables $(y,z)$, featuring unbounded stochastic coefficients $\mu_\cdot\in\R$ and $\nu_\cdot\in\R_+$ satisfying $\int_0^\tau (|\mu_t|+\nu^2_t) {\rm d}t<+\infty$. For any given real $p>1$, let $\rho_\cdot\geq \mu_\cdot+\frac{\theta}{2(p-1)}\nu_\cdot^2$ (instead of $\rho_\cdot\geq\mu_\cdot+\frac{\theta}{2[1\wedge(p-1)]}\nu_\cdot^2$ used in Zhang, Li, Hu and Fan [2026, arXiv:2603.13873v1]) be a real-valued process for some constant $\theta>1$ such that $\int_0^\tau |\rho_t|{\rm d}t<+\infty$. We work within a weighted $L^p$ space with the weighting factor $e^{\int_0^t \rho_r{\rm d}r}$. Within this framework, we establish several innovative results on the weighted $L^p$ solutions of BSDEs: an existence result, an existence and uniqueness result, an existence and uniqueness result of the minimal (maximal) solution, and two comparison theorems. These findings unify and improve some existing results. Some novel ideas are employed to address the challenges posed by general unbounded stochastic coefficients and general weighted spaces.
\vspace{0.2cm}
\end{abstract}

\begin{keyword}
Backward stochastic differential equation \sep Weighted $L^p$ solution \sep Existence and uniqueness\sep\\  \hspace*{1.85cm} Comparison theorem \sep Unbounded stochastic coefficients \sep Minimal solution\vspace{0.2cm}

\MSC[2021] 60H10, 60H30\vspace{0.2cm}
\end{keyword}

\end{frontmatter}
\vspace{-0.4cm}

\section{Introduction}
\label{sec:1-Introduction}
\setcounter{equation}{0}

Let us fix a positive integer $d$ and let $(\Omega,\F, \mathbb{P})$ be a completed probability space with the natural augmented filtration $(\F_t)_{t\geq0}$ generated by a standard $d$-dimensional Brownian motion $(B_t)_{t\geq0}$. We always assume that $\tau$ is a terminal time being a general $(\F_t)$-stopping time whose values are either nonnegative real numbers or positive infinity, and set $\F:=\F_\tau$. In the case of $\tau(\omega)=+\infty$, the interval $[0,\tau(\omega)]$ should be always understood as $[0,+\infty)$. In this paper, we are concerned with the following one-dimensional backward stochastic differential equation (BSDE for
short in the remaining):
\begin{align}\label{BSDE1.1}
  y_t=\xi+\int_t^\tau g(s,y_s,z_s){\rm d}s-\int_t^\tau z_s\cdot{\rm d}B_s, \ \ t\in[0,\tau],
\end{align}
where the terminal value $\xi$ is an $\F_\tau$-measurable random variable and the random function
$$g(\omega,t,y,z): \Omega\times[0,\tau]\times \R\times \R^{d}\mapsto \R$$
is $(\F_t)$-progressively measurable for each $(y,z)$, called the generator of BSDE \eqref{BSDE1.1}. A pair of $(\F_t)$-progressively measurable processes $(y_t,z_t)_{t\in [0,\tau]}$ taking values in $\R \times \R^{d}$ is called an adapted solution to BSDE \eqref{BSDE1.1} if $\ps$, $y_\cdot$ is continuous, $\int_0^{\tau} \left(|g(t,y_t,z_t)|+|z_t|^2\right){\rm d}t<+\infty$, and \eqref{BSDE1.1} is satisfied. Notably, BSDE \eqref{BSDE1.1} indicates that
$\lim\limits_{t\rightarrow \tau} y_t=\xi$ on the set of $\{\tau=+\infty\}$. The BSDE with parameters $(\xi,\tau,g)$ will be denoted by BSDE$(\xi,\tau,g)$.

Since the pioneering work of \citet{PardouxPeng1990SCL} on nonlinear BSDEs with a fixed terminal time and square-integrable parameters under a uniform Lipschitz condition on the generator $g$, the theory of BSDEs has been extensively developed and applied in various fields such as partial differential equations (\cite{DarlingandPardoux1997,Pardoux1999,BriandandConfortola2008,Bahlali2015}), mathematical finance
(\cite{KarouiPengQuenez1997,HuYing2005,DelbaenPengRosazza Gianin2010,Tian2023SIAM}), and stochastic control (\cite{Yong1999,Kohlmann2003,Tang2003,PardouxandRascanu2014}).
Many studies have sought to extend the existence and uniqueness result of \citet{PardouxPeng1990SCL} by relaxing the uniform Lipschitz condition on the generator $g$, considering BSDEs over random (infinite) terminal time, and investigating adapted solutions of BSDEs under non-square integrable parameters. Relevant developments can be found in \cite{Kobylanski2000,Briand2003,Bahlali2004,Royer2004,Fan2016SPA} and references therein.

For the one-dimensional case, a lot of works have been made by virtue of the comparison theorem of solutions to one-dimensional BSDEs initially proposed by \citet{Peng1992}, and then established in \cite{Lepeltier1998,Jia2010,DelbaenHuBao2011,FanHu2021SPA,Gu2025} and so on. We are particularly interested in the study of one-dimensional BSDEs with a random terminal time $\tau$ and with a generator satisfying some stochastic growth/continuity conditions. For convenience of narration, let $p,\theta>1$ be given constants and let $\rho_\cdot, \mu_\cdot\in \R$ and $\nu_\cdot\in\R_+$ be $(\F_t)$-progressively measurable processes defined on $\Omega\times[0,\tau]$ and satisfying $\mathbb{P}-a.s.$, $\int_0^\tau \left(|\rho_t|+|\mu_t|+\nu_t^2\right){\rm d}t<+\infty$. Very recently, under a stochastic monotonicity condition and a stochastic Lipschitz condition of $g$ in $(y,z)$, i.e., $\mathbb{P}-a.s.$, for each $(t,y_1,y_2,z_1,z_2)\in [0,\tau]\times \R\times\R\times\R^{d}\times\R^{d}$,
\begin{align}\label{eq:1.1}
\sgn(y_1-y_2)(g(\omega,t,y_1,z)-g(\omega,t,y_2,z))\leq \mu_t(\omega)|y_1-y_2|
\end{align}
and\vspace{-0.1cm}
\begin{align}\label{eq:1.2}
\left|g(\omega,t,y,z_1)-g(\omega,t,y,z_2)\right|\leq \nu_t(\omega)|z_1-z_2|,
\end{align}
combined with the following general growth condition of $g$ in $y$:
there exists an $(\F_t)$-progressively measurable real-valued process $(\alpha_t)_{t\in [0,\tau]}$ satisfying $\essinf\limits_{t\in[0,\tau]}\alpha_t>0$ such that
\begin{align}\label{eq:1.3}
\forall x\in \R_+,\ \ \ \E\left[\int_0^\tau \alpha_t\psi_t(x\alpha_t){\rm d}t\right]<+\infty
\end{align}
with \vspace{-0.1cm}
\begin{align}\label{eq:1.4}
\psi_t(x):=\sup_{|y|\leq x} \left|g(t,y,0)-g(t,0,0)\right|, \ \ (t,x)\in [0,\tau]\times\R_+,
\end{align}
\citet[Theorem 3.3]{Zhang2026} proved that if
\begin{align}\label{eq:1.6}
\rho_t\geq \mu_t +\frac{\theta}{2[1\wedge(p-1)]}\nu_t^2,\ \ t\in [0,\tau]
\end{align}
and
\begin{align}\label{eq:1.5}
\E\left[\left(\int_0^\tau e^{\int_0^s \rho_r {\rm d}r} |g(s,0,0)|{\rm d}s\right)^{p}\right]<+\infty,
\end{align}
then for each $\xi \in L_\tau^p(\rho_\cdot;\R)$, BSDE$(\xi,\tau,g)$ admits a unique adapted solution $(y_t,z_t)_{t\in [0,\tau]}$ in the weighted space of $H_\tau^p(\rho_\cdot;\R\times\R^{d})$ with the weighting factor $e^{\int_0^t \rho_r {\rm d}r}$ (see the definitions of $L_\tau^p(\rho_\cdot;\R)$, $S_\tau^p(\rho_\cdot;\R)$, $M_\tau^p(\rho_\cdot;\R^d)$ and $H_\tau^p(\rho_\cdot;\R\times\R^{d})$ in Section 2.1). We note that \citet[Theorem 3.3]{Zhang2026} actually dealt with multidimensional BSDEs. As for the following one-dimensional linear BSDE:
\begin{align}\label{linearBSDE}
  y_t=\xi+\int_t^\tau (\mu_sy_s+\nu_sz_s){\rm d}s-\int_t^\tau z_s\cdot {\rm d}B_s, \ \ t\in[0,\tau],
\end{align}
\citet[subsection 2.2]{Zhang2026} illustrates that if
\begin{align}\label{eq:1.7}
\rho_t\geq \mu_t+\frac{{\theta}}{2(p-1)}\nu_t^2,\ \ t\in [0,\tau],
\end{align}
then for each $\xi \in L_\tau^p(\rho_\cdot;\R)$, BSDE \eqref{linearBSDE} still admits an adapted solution $(y_t,z_t)_{t \in[0,\tau]}$ such that $y_\cdot\in S_\tau^p(\rho_\cdot;\R)$, although it is uncertain that $z_\cdot\in M_\tau^p(\rho_\cdot;\R)$. We emphasize that for the case of $p>2$, the condition in \eqref{eq:1.7} required for the process $\rho_\cdot$ is weaker than that in \eqref{eq:1.6}. Then, a question naturally arises: for a scalar nonlinear BSDE with general unbounded stochastic coefficients, is the previous assertion also true? This paper focuses on solving this question and establishes several affirmative results.

Let us proceed to review several known results on existence of an adapted solution to one-dimensional BSDEs, which are related closely to the research of this paper. \citet{Lepeltier1997} proved existence of the minimal (maximal) $L^2$ solution for a one-dimensional BSDE with a positive constant terminal time $T$ when the generator $g$ is continuous and has a linear growth in $(y,z)$, which was extended by \citet{Briand2007} to the case that $g$ satisfies \eqref{eq:1.1} with a constant $\mu\geq0$ instead of the process $\mu_\cdot\in\R$ and it has a general growth in $y$. Furthermore, \citet{Fan2016SPL} first proposed a one-sided linear growth condition in $(y,z)$, i.e., there exist a constant $\mu\geq0$ and a nonnegative $(\F_t)$-progressively measurable process $(f_t)_{t\in [0,T]}$ satisfying
$\E\left[\left(\int_0^T f_t{\rm d}t\right)^{2}\right]<+\infty$
such that $\mathbb{P}-a.s.$, for each $(t,y,z) \in [0,\tau]\times\R\times\R^d$,
\begin{align}\label{eq:1.10}
\sgn(y)g(\omega,t,y,z) \leq f_t(\omega) +\mu(|y|+|z|),
\end{align}
and established existence of an $L^2$ solution for a one-dimensional BSDE under some additional conditions.

Inspired by the above-mentioned works, in this paper we establish several novel results on the weighted $L^p$ solutions of scalar BSDEs with general unbounded stochastic coefficients: an existence result, an existence and uniqueness result, an existence and uniqueness result of the minimal (maximal) solution, and two comparison theorems. More specifically, in \cref{thm:3.2} of \cref{sec:3-An existence result}, we suppose that \eqref{eq:1.7} holds and that the generator $g$ is continuous and satisfies the following stochastic one-sided linear growth condition in $(y,z)$: there exists a nonnegative $(\F_t)$-progressively measurable process $(f_t)_{t\in [0,\tau]}$ satisfying
$\E\left[\left(\int_0^\tau e^{\int_0^s \rho_r {\rm d}r} f_t{\rm d}t\right)^p\right]<+\infty$ such that $\mathbb{P}-a.s.$, for each $(t,y,z) \in [0,\tau]\times\R\times\R^d$,
\begin{align}\label{eq:1.11}
\sgn(y)g(\omega,t,y,z) \leq f_t(\omega) +\mu_t(\omega)|y|+\nu_t(\omega)|z|,
\end{align}
along with a general growth condition in $y$ and a quadratic growth condition in $z$ (see \ref{A:H1} in \cref{sec:3-An existence result} for details), and we prove that if either $\mu_\cdot\in \R_+$ is satisfied or the quadratic growth condition of $g$ in $z$ is replaced with a stochastic linear growth condition (see \eqref{eq:2.3*} for details), then for each $\xi \in L_\tau^p(\rho_\cdot;\R)$,
BSDE$(\xi,\tau,g)$ admits an adapted solution $(y_t,z_t)_{t\in[0,\tau]}$ such that $y_\cdot\in S_\tau^p(\rho_\cdot;\R)$, and $z_\cdot\in M_\tau^p(\rho_\cdot;\R^d)$ when \eqref{eq:1.6} is in force. It is obvious that condition \eqref{eq:1.11} generalizes condition \eqref{eq:1.10}. We emphasize that \cref{thm:3.2} unifies and strengthens the corresponding results obtained in \citet{Lepeltier1997}, \citet{Briand2007} and \citet{Fan2016SPL} as well as some existing relevant results. Subsequently, in \cref{thm:4.2} of \cref{sec:4-An existence and uniqueness result}, we suppose that the generator $g$ is continuous in $y$ and satisfies \eqref{eq:1.1}-\eqref{eq:1.2}, \eqref{eq:1.5} and the following general growth condition in $y$: for the function $\psi$ defined in \eqref{eq:1.4}, we have
\begin{align}\label{eq:1.12}
\RE x\in \R_+, \ \ \mathbb{P}-a.s.,\ \ \ \int_0^\tau \psi_t(x){\rm d}t<+\infty
\end{align}
(see \ref{A:H2} in \cref{sec:4-An existence and uniqueness result} for details), and we prove that if \eqref{eq:1.7} holds, then for each $\xi \in L_\tau^p(\rho_\cdot;\R)$, BSDE$(\xi,\tau,g)$ admits a unique adapted solution $(y_t,z_t)_{t\in[0,\tau]}$ such that $y_\cdot\in S_\tau^p(\rho_\cdot;\R)$, and $z_\cdot\in M_\tau^p(\rho_\cdot;\R^d)$ when \eqref{eq:1.6} is in force. Since \eqref{eq:1.7} is weaker than \eqref{eq:1.6}, and \eqref{eq:1.12} is weaker than \eqref{eq:1.3}, it can be concluded that \cref{thm:4.2} strengthens Theorem 3.3 in \citet{Zhang2026} in the one-dimensional case. Moreover, by using the weighting factor $e^{\int_0^t \rho_r {\rm d}r}$ with process $\rho_\cdot$ satisfying \eqref{eq:1.7} rather than $\rho_\cdot:=\beta \mu_\cdot^+ +\frac{\theta}{2[1\wedge(p-1)]} \nu_\cdot^2$ for some $\beta\geq1$ used in \citet{LiFan2024SD}, we establish a general comparison theorem on the weighted $L^p$ solution of scalar BSDEs under \eqref{eq:1.1} and \eqref{eq:1.2} (see \cref{thm:4.1} for details) and it naturally improves the corresponding result obtained in \citet{LiFan2024SD}. Finally, in \cref{thm:5.1} of \cref{sec:5-An existence and uniqueness result of the minimal (maximal) solution}, we suppose that \eqref{eq:1.7}/\eqref{eq:1.6} holds, either the generator $g$ is continuous in $y$ (uniformly with respect to $z$) and $z$, and it satisfies the stochastic monotonicity condition with general growth in $y$ and the stochastic linear growth condition in $z$ (see \ref{A:H3} in \cref{sec:5-An existence and uniqueness result of the minimal (maximal) solution} for details) or the generator $g$ is continuous and has a stochastic linear growth condition in $(y,z)$ (see \ref{A:H1'} in \cref{sec:3-An existence result} for details), and we prove an existence and uniqueness result on the minimal (maximal) weighted $L^p$ solution of BSDE$(\xi,\tau, g)$, which strengthens Theorem 4.1 of \citet{Briand2007} and some relevant results. The corresponding comparison theorem of the minimal (maximal) weighted $L^p$ solutions is also established in \cref{thm:5.2} of \cref{sec:5-An existence and uniqueness result of the minimal (maximal) solution}.

A key a priori estimate on the weighted $L^p$ solution of BSDEs with general unbounded stochastic coefficients is established in \cref{pro:2.2} of \cref{sec:2-Preliminaries} in order to tackle the general case that the process $\rho_\cdot$ only satisfies condition \eqref{eq:1.7}. To handle the case that the stochastic process $\mu_\cdot$ may take negative values, in the proof of \cref{thm:3.2} we employ an exponential shift transformation, a truncation technique and a localization method. Due to general unbounded stochastic coefficients, some new difficulties arise naturally and are successfully overcome in the proof of \cref{thm:4.1}. \cref{thm:4.2} follows immediately by combining \cref{thm:3.2} and \cref{thm:4.1}. Based on \cref{thm:4.1} and \cref{thm:4.2}, we can prove \cref{thm:5.1} by virtue of a convolution approaching technique, the localization method and the a priori estimate-\cref{pro:2.2}. \cref{thm:5.2} is a direct consequence of \cref{thm:5.1} and \cref{thm:4.1}. It should be emphasized that, under condition \eqref{eq:1.7} rather than \eqref{eq:1.6}, the proof of the existence and uniqueness results for the weighted $L^p$ solution of the aforementioned one-dimensional BSDE crucially relies on the a priori estimate provided in \cref{pro:2.2} and the localization method presented initially in \citet{BriandHu2006PTRF}. These proof methods and techniques are no longer applicable to the case of multidimensional BSDEs. Consequently, whether the corresponding multidimensional versions of the results obtained in this paper still hold remains unknown. This problem is both interesting and somewhat challenging.

The rest of this paper is organized as follows. Section 2 introduces some notations and establishes two a priori estimates (\cref{pro:2.2,pro:2.3}). Section 3 presents an existence result (\cref{thm:3.2}) and its direct corollary (\cref{cor:3.1}), along with their proofs. Section 4 provides an existence and uniqueness result and a comparison theorem (\cref{thm:4.2} and \cref{thm:4.1}), along with their proofs. Section 5 establishes an existence and uniqueness result for the minimal (maximal) solutions (\cref{thm:5.1}) and the corresponding comparison theorem (\cref{thm:5.2}). The Appendix contains the proofs of two technical propositions (\cref{pro:5.2,pro:5.1}). Some remarks and examples are also provided after each theorem to compare our results with existing works and to illustrate the novelty of our findings.

\section{Preliminaries}
\label{sec:2-Preliminaries}
\setcounter{equation}{0}

\subsection{Notations}
For $a,b\in \R$, define $a\vee b=\max\{a,b\}$, $a\wedge b=\min\{a,b\}$, $a^+=a\vee 0$ and $a^-=(-a)^+$. For each positive integer $n$, denote the norm of Euclidean space $\R^n$ by $|\cdot|$. Let $\R_+:=[0,+\infty)$ and $\sgn(x)={\bf 1}_{x>0}-{\bf 1}_{x\leq0}$, where ${\bf 1}_{A}$ is the indicator function of $A$. Every equality and inequality between random elements and every claim on random elements should be understood as holding $\mathbb{P}$-almost surely. Denote by $\mathbf{S}$ the set of functions $\phi_t(\omega,x): \Omega \times [0,\tau] \times \R_+ \rightarrow \R_+$ satisfying the following two conditions:
\begin{itemize}
\item $\as$, the function $x \mapsto \phi_t(\omega,x)$ is continuous and increasing;
\item $\forall x \in \R_+$, $\phi_t(\omega,x)$ is an $(\F_t)$-progressively measurable process satisfying
$\int_0^\tau \phi_t(\omega,x){\rm d}t < +\infty.$
\end{itemize}

Let $p>1$ be a given constant and $\rho_t(\omega):\Omega \times [0,\tau]\mapsto \R$ be an $(\F_t)$-progressively measurable process satisfying $\int_{0}^{\tau}|\rho_t|{\rm d}t<+\infty$. Let $L_\tau^p(\rho_\cdot;\R)$ denote the set of $\F_\tau$-measurable random variables $\xi$ such that $$\|\xi\|_{p;\rho_\cdot}:=\left(\E\left[e^{p \int_0^\tau \rho_s{\rm d}s}|\xi|^p\right]\right)^{\frac{1}{p}}<+\infty,$$
$S_\tau^p(\rho_\cdot;\R)$ the set of $(\F_t)$-adapted and continuous real-valued processes $(Y_t)_{t\in[0,\tau]}$ such that
$$\|Y_\cdot\|_{p; \rho_\cdot,c}:=\left(\E\left[\sup_{t\in[0,\tau]}\left(e^{p \int_0^t \rho_r{\rm d}r}|Y_t|^p\right)\right]\right)^{\frac{1}{p}}<+\infty,$$
$S^\infty_\tau(\rho_\cdot;\R)$ the subset of $S_\tau^p(\rho_\cdot;\R)$ such that the process $e^{\int_0^\cdot \rho_r{\rm d}r}|Y_\cdot|$ is bounded, and $M_\tau^p(\rho_\cdot;\R^{d})$ the set of $(\F_t)$-progressively measurable $\R^{d}$-valued processes $(Z_t)_{t\in[0,\tau]}$ such that
$$\|Z_\cdot\|_{p; \rho_\cdot}:=\left(\E\left[{\left(\int_0^\tau e^{2 \int_0^t \rho_r{\rm d}r}|Z_t|^2{\rm d}t\right)}^\frac{p}{2}\right]\right)^{\frac{1}{p}}<+\infty.$$ Then, define
$$H_\tau^p(\rho_\cdot;\R\times\R^{d}):=S_\tau^p(\rho_\cdot;\R)\times M_\tau^p(\rho_\cdot;\R^{d}).$$
It is clear that $H_\tau^p(\rho_\cdot;\R\times\R^{d})$ is a Banach space with the norm
$$\|(Y_\cdot,Z_\cdot)\|_{p; \rho_\cdot}:=\|Y_\cdot\|_{p; \rho_\cdot,c}
+\|Z_\cdot\|_{p; \rho_\cdot}.$$
Finally, if a pair of processes $(y_t,z_t)_{t\in [0,\tau]}$ is an adapted solution of BSDE \eqref{BSDE1.1} such that $y_\cdot\in S_\tau^p(\rho_\cdot;\R)$, then it will be called a weighted $L^p$ solution of BSDE \eqref{BSDE1.1} with the weighting factor $e^{\int_0^t \rho_s{\rm d}s}$.

\subsection{A priori estimates}

In this subsection, we establish the following a priori estimate concerning the weighted $L^p$ solution of BSDE \eqref{BSDE1.1} that will play an important role in the proof of main results in this paper.

\begin{pro}\label{pro:2.2}
Let $\bar\rho_t(\omega)$, $\bar\mu_t(\omega)$: $\Omega \times [0,\tau]\mapsto \R$ and $\bar f_t(\omega)$, $\bar\nu_t(\omega)$: $\Omega \times [0,\tau]\mapsto \R_+$ be four $(\F_t)$-progressively measurable processes satisfying $\int_{0}^{\tau}\left(|\bar\rho_t|+|\bar\mu_t|+\bar\nu_t^2\right){\rm d}t<+\infty$,
\begin{align*}
\bar\rho_t\geq \bar\mu_t+\frac{\theta}{2(p-1)}\bar\nu_t^2, \ \ t\in[0,\tau]
\end{align*}
for some constant $\theta>1$, and
$$
\E\left[\left(\int_0^\tau e^{\int_{0}^{t}\bar\rho_r{\rm d}r}\bar f_t{\rm d}t\right)^{p}\right]<+\infty.
$$
Assume that the generator $g$ satisfies the following assumption:
\begin{enumerate}
\renewcommand{\theenumi}{(A)}
\renewcommand{\labelenumi}{\theenumi}
\item\label{A:A} $\forall (t,y,z)\in[0,\tau]\times\R\times\R^d$,~ $\sgn(y)g(\omega,t,y,z) \leq \bar f_t(\omega) + \bar\mu_t(\omega)|y| + \bar\nu_t(\omega)|z|$,
\end{enumerate}
and $(y_t,z_t)_{t\in [0,\tau]}$ is an adapted solution of BSDE \eqref{BSDE1.1} such that $y_\cdot\in S_\tau^{p}(\bar\rho_\cdot;\R)$. Then there exists a constant $C_{p,\theta}>0$ depending only on $p$ and $\theta$ such that for each $0\leq r\leq t< +\infty$,
\begin{align*}
\E\left[\sup_{s\in[t\wedge\tau,\tau]}\left(e^{p
\int_{0}^{s}\bar{\rho}_{r}{\rm d}r}|y_s|^{p}\right)\bigg|\F_{r\wedge\tau}\right]\leq
C_{p,\theta}\E\left[e^{p\int_{0}^{\tau}\bar{\rho}_{r}{\rm d}r}|\xi|^{p}+\left(
\int_{t\wedge\tau}^{\tau}e^{ \int_{0}^{s}\bar{\rho}_r{\rm d}r}\bar{f}_s{\rm d}s\right)^{p}\bigg|\F_{r\wedge\tau}\right].
\end{align*}
\end{pro}
\begin{proof}[\bf Proof]
The proof can be regarded as a slight modification of Proposition 2.3 in \citet{Zhang2026}, where $\bar\rho_t\geq \bar\mu_t+\frac{\theta}{2[1\wedge(p-1)]}\bar\nu_t^2$ for each $t\in[0,\tau]$, and then $z_\cdot\in M_\tau^{p}(\bar\rho_\cdot;\R^{d})$ (see \cref{pro:2.3} below for details). For readers' convenience, we list it as follows. Define $y_t^{\prime}=e^{\int_{0}^{t}\bar{\rho}_r{\rm d}r}y_t$, $z_t^{\prime}=e^{\int_{0}^{t}\bar{\rho}_r{\rm d}r}z_t$ and $f_t^{\prime}=e^{\int_{0}^{t}\bar{\rho}_r{\rm d}r}\bar{f}_t$ for each $t\in [0,\tau]$. Then $(y_t^{\prime})_{t\in[0,\tau]}\in S_\tau^{p}(0;\R)$, and it suffices to prove that there exists a constant $C_{p,\theta} >0$ depending only on $p$ and $\theta$ such that for each $0\leq r\leq t< +\infty$,
\begin{align}\label{6.1}
\E\left[\sup_{s\in[t\wedge\tau,\tau]}|y_s^{\prime}|^{p}\bigg|
\F_{r\wedge\tau}\right]\leq
C_{p,\theta}\E\left[|\xi^{\prime}|^{p}+\left(
\int_{t\wedge\tau}^{\tau}f_s^{\prime}{\rm d}s\right)^{p}\bigg|\F_{r\wedge\tau}\right].
\end{align}
For each integer $n\geq1$, define the following $(\F_t)$-stopping time
\begin{align}\label{6.2}
\tau_{n}:=\inf \left\{t \geq0: \int_{0}^{t}|z_s^{\prime}|^{2} {\rm d}s \geq n\right\}\wedge \tau,
\end{align}
with the convention that $\inf \emptyset=+\infty$. In light of assumption \ref{A:A}, applying It\^{o}-Tanaka's formula to $|y_t^{\prime}|^{p}$ yields that for each $t\geq0$ and $n\geq1$,
\begin{align}\label{6.3}
\begin{split}
&|y^{\prime}_{t\wedge \tau_n}|^{p}+c(p)\int_{t\wedge \tau_n}^{\tau_n} |y^{\prime}_s|^{p-2}{\bf 1}_{|y^{\prime}_s|\neq0}|z^{\prime}_s|^2{\rm d}s+p \int_{t\wedge\tau_n}^{\tau_n} \bar{\rho}_s|y^{\prime}_s|^{p}{\rm d}s\\
& \ \ \leq |y^{\prime}_{\tau_n}|^{p}+p\int_{t\wedge\tau_n}^{\tau_n} (\bar\mu_{s}|y^{\prime}_s|^{p}+\bar\nu_s
|y^{\prime}_s|^{p-1}|z^{\prime}_s|+|y^{\prime}_s|^{p-1}f^{\prime}_{s}){\rm d}s
-p\int_{t\wedge\tau_n}^{\tau_n} |y^{\prime}_s|^{p-2}{\bf 1}_{|y^{\prime}_s|\neq0}y^{\prime}_s z^{\prime}_s\cdot{\rm d}B_s,
\end{split}
\end{align}
where $c(p):=\frac{p(p-1)}{2}$.

We can verify that $\{M_t=\int_{0}^{t} |y^{\prime}_s|^{p-2}{\bf 1}_{|y^{\prime}_s|\neq0}y^{\prime}_s z^{\prime}_s\cdot{\rm d}B_s\}_{t\in[0,\tau_n]}$ is a uniformly integrable martingale for each $n\geq1$. Indeed, by virtue of \eqref{6.2} and $(y^{\prime}_t)_{t\in[0,\tau]} \in S_\tau^{p}(0;\R)$, the BDG
inequality (see Theorem 1 in \citet{Ren2008BDG}) and Young's inequality, we can obtain that for each $n\geq1$,
\begin{align}\label{6.4}
\begin{split}
&\E\left[\sup_{t\in[0,\tau_n]}\bigg|\int_{0}^{t} |y^{\prime}_s|^{p-2}{\bf 1}_{|y^{\prime}_s|\neq0}y^{\prime}_s z^{\prime}_s\cdot{\rm d}B_s\bigg|\right]
\leq 2\sqrt{2}\E\left[\left(\int_{0}^{\tau_n}|y^{\prime}_s|^{2p-2}|z^{\prime}_{s}|^2 {\rm d}{s}\right)^{\frac{1}{2}}\right]\\
&\ \ \leq 2\sqrt{2}\E\left[\left(\sup_{s\in[0,\tau_n]}|y^{\prime}_s|^{p-1}\right)
\left(\int_{0}^{\tau_n}|z^{\prime}_{s}|^2{\rm d}{s}\right)^{\frac{1}{2}}\right]\\
&\ \ \leq \frac{2\sqrt{2}(p-1)}{p}\E\left[\sup_{s\in[0,\tau]}
|y^{\prime}_s|^{p}\right]+\frac{2\sqrt{2}}{p}
\E\left[\left(\int_{0}^{\tau_n}|z^{\prime}_{s}|^2{\rm d}{s}\right)^{\frac{p}{2}}\right]<+\infty.
\end{split}
\end{align}
Moreover, in light of assumption \ref{A:A}, by Young's inequality and H\"{o}lder's inequality we deduce that
\begin{align}\label{6.5}
\begin{split}
&\int_0^{\tau_n}(|\bar{\mu}_s||y^{\prime}_s|^{p}+\bar\nu_s
|y^{\prime}_s|^{p-1}|z^{\prime}_s|+|y^{\prime}_s|^{p-1}f^{\prime}_{s}){\rm d}s\\
&\ \ \leq  \left(\sup\limits_{s\in [0,\tau_n]}|y^{\prime}_s|^{p}\right)\int_0^{\tau_n} |\bar{\mu}_s|{\rm d}s+\left(\sup\limits_{s\in [0,\tau_n]}|y^{\prime}_s|^{p-1}\right) \int_0^{\tau_n} \bar\nu_s |z^{\prime}_s|{\rm d}s+ \left(\sup\limits_{s\in [0,\tau_n]}|y^{\prime}_s|^{p-1}\right) \int_0^{\tau_n} f^{\prime}_{s}{\rm d}s\\
&\ \ \leq \left(\sup\limits_{s\in [0,\tau]}|y^{\prime}_s|^{p}\right)\int_0^\tau |\bar{\mu}_s|{\rm d}s +\frac{2(p-1)}{p}\sup\limits_{s\in [0,\tau]}|y^{\prime}_s|^{p}\\
&\qquad +\frac{1}{p}\left(\int_0^\tau v_s^2{\rm d}s \right)^{\frac{p}{2}}\left(\int_0^{\tau_n} |z^{\prime}_s|^2{\rm d}s \right)^{\frac{p}{2}}+\frac{1}{p}\left(\int_0^\tau f^{\prime}_s{\rm d}s\right)^{p}<+\infty.
\end{split}
\end{align}
Combining \eqref{6.3}-\eqref{6.5} yields that for each $t\geq0$ and $n\geq1$,
\begin{align}\label{6.6}
\int_{t\wedge\tau_n}^{\tau_n} |y^{\prime}_s|^{p-2}{\bf 1}_{|y^{\prime}_s|\neq0}|z^{\prime}_s|^2{\rm d}s<+\infty.
\end{align}
Furthermore, using the inequality $p ab\leq \frac{p^2 \theta}{4c(p)} a^2+\frac{c(p)}{\theta}b^2$ we obtain
\begin{align}\label{6.7}
p\bar\nu_s|y^{\prime}_s|^{p-1}|z^{\prime}_s|=p(\bar\nu_s
|y^{\prime}_s|^{\frac{p}{2}})(|y^{\prime}_s|^{\frac{p}{2}-1}|z^{\prime}_s|)
\leq \frac{p^2 \theta}{4c(p)} \bar\nu_{s}^2|y^{\prime}_s|^{p}+\frac{c(p)}{\theta}
|y^{\prime}_s|^{p-2}{\bf 1}_{|\bar{y}_s|\neq0}|z^{\prime}_s|^2, \ \ s\in[0,\tau_n].
\end{align}
In light of \eqref{6.6}, putting \eqref{6.7} into \eqref{6.3} yields that for each $t\geq0$ and $n\geq1$,
\begin{align}\label{6.8}
\begin{split}
&|y^{\prime}_{t\wedge\tau_n}|^{p}+(1-\frac{1}{\theta})c(p)
\int_{t\wedge\tau_n}^{\tau_n} |y^{\prime}_s|^{p-2}{\bf 1}_{|y^{\prime}_s|\neq0}|z^{\prime}_s|^2{\rm d}s
+p \int_{t\wedge\tau_n}^{\tau_n} \left[\bar{\rho}_s-(\bar{\mu}_s+\frac{\theta}{2(p-1)}
{\bar\nu_s}^2)\right]|y^{\prime}_s|^{p}{\rm d}s\\
&\ \ \leq |y^{\prime}_{\tau_n}|^{p}+p\int_{t\wedge\tau_n}^{\tau_n} |y^{\prime}_s|^{p-1}f^{\prime}_s{\rm d}s
-p\int_{t\wedge\tau_n}^{\tau_n} |y^{\prime}_s|^{p-2}{\bf 1}_{|y^{\prime}_s|\neq0}y^{\prime}_s z^{\prime}_s\cdot{\rm d}B_s.
\end{split}
\end{align}
Then, in light of \eqref{6.4}, taking the conditional mathematical expectation with respect to $\F_{r\wedge\tau_m}$ on both sides of \eqref{6.8} yields that for each $0\leq r\leq t<+\infty$ and $n\geq m\geq1$,
\begin{align}\label{6.9}
\E\left[\int_{t\wedge\tau_n}^{\tau_n} |y^{\prime}_s|^{p-2}{\bf 1}_{|y^{\prime}_s|\neq0}|z^{\prime}_s|^2{\rm d}s\bigg|\F_{r\wedge\tau_m}\right]\leq \frac{\theta}{(\theta-1)c(p)}\E\left[|y^{\prime}_{\tau_n}|^{p}+p
\int_{t\wedge\tau_n}^{\tau_n} |y^{\prime}_s|^{p-1}f^{\prime}_s{\rm d}s\bigg|\F_{r\wedge\tau_m}\right].
\end{align}

On the other hand, by virtue of the conditional BDG inequality in \cite[Theorem 1]{Ren2008BDG} and the elementary inequality $2ab\leq a^2+b^2$, we also deduce that for each $0\leq r\leq t<+\infty$ and $n\geq m\geq1$,
\begin{align*}
&p\E\left[\sup_{u \in [t\wedge\tau_n,\tau_n]}\bigg|\int_{u}^{\tau_n}|y^{\prime}_s|^{p-2}{\bf 1}_{|y^{\prime}_s|\neq0}y^{\prime}_s z^{\prime}_s\cdot{\rm d}B_s\bigg| \bigg|\F_{r\wedge\tau_m}\right]\\
&\ \  \leq
2\sqrt{2}p\E\left[\left(\int_{t\wedge\tau_n}^{\tau_n}|y^{\prime}_s|^{2p-2}{\bf 1}_{|y^{\prime}_s|\neq0}|z^{\prime}_{s}|^2 {\rm d}{s}\right)^{\frac{1}{2}}\bigg|\F_{r\wedge\tau_m}\right]\\
&\ \ \leq 2\sqrt{2}p\E\left[\left(\sup_{s\in[t\wedge\tau_n,\tau_n]}
|y^{\prime}_s|^{\frac{p}{2}}\right)\left(\int_{t\wedge\tau_n}^{\tau_n}
|y^{\prime}_s|^{p-2}{\bf 1}_{|y^{\prime}_s|\neq0}|z^{\prime}_{s}|^2{\rm d}{s}\right)^{\frac{1}{2}}\bigg|\F_{r\wedge\tau_m}\right]\\
&\ \ \leq \frac{1}{2}\E\left[\sup_{s\in[t\wedge\tau_n,\tau_n]}|y^{\prime}_s|^{p}
\bigg|\F_{r\wedge\tau_m}\right]+4p^2
\E\left[\int_{t\wedge\tau_n}^{\tau_n}|y^{\prime}_s|^{p-2}{\bf 1}_{|y^{\prime}_s|\neq0}|z^{\prime}_{s}|^2{\rm d}{s}\bigg|\F_{r\wedge\tau_m}\right].
\end{align*}
Then, in light of the last inequality and the condition that $\bar{\rho}_t\geq \bar{\mu}_t+\frac{\theta}{2(p-1)}{\bar\nu}_t^2$ along with \eqref{6.8} and \eqref{6.9}, we have that for each $0\leq r\leq t< +\infty$ and $n\geq m\geq1$,
\begin{align}\label{6.10}
\E\left[\sup_{s\in[t\wedge\tau_n,\tau_n]}|y^{\prime}_s|^{p}\bigg|
\F_{r\wedge\tau_m}\right]
\leq \left(2+\frac{8p^2 \theta}{(\theta-1)c(p)}\right)\E\left[|y^{\prime}_{\tau_n}|^{p}
+p\int_{t\wedge\tau_n}^{\tau_n} |y^{\prime}_s|^{p-1}f^{\prime}_s{\rm d}s\bigg|\F_{r\wedge\tau_m}\right].
\end{align}
Letting $M_{p,\theta}:=p(2+\frac{8p^2 \theta}{(\theta-1)c(p)})$, by using Young's inequality we deduce that for each $0\leq r\leq t<+\infty$ and $n\geq m\geq1$,
\begin{align}\label{6.11}
\begin{split}
&M_{p,\theta}\E\left[\int_{t\wedge\tau_n}^{\tau_n} |y^{\prime}_s|^{p-1}f^{\prime}_s{\rm d}s\bigg|\F_{r\wedge\tau_m}\right]\\
&\ \ \leq \E\left[\sup_{s\in[t\wedge\tau_n,\tau_n]}\left(\left(\frac{p}{2p-2}\right)^{\frac{p-1}{p}}|y^{\prime}_s|^{p-1}\right)\left(M_{p,\theta}\left(\frac{p}{2p-2}\right)^{\frac{1-p}{p}}\int_{t\wedge\tau_n}^{\tau_n} f^{\prime}_s{\rm d}s\right)\bigg|\F_{r\wedge\tau_m}\right]\\
&\ \ \leq \frac{1}{2}\E\left[\sup_{s\in[t\wedge\tau_n,\tau_n]}|y^{\prime}_s|^{p}
\bigg|\F_{r\wedge\tau}\right]+\frac{M_{p,\theta}}{p}
\left(\frac{p}{2p-2}\right)^{1-p}\E\left[
\left(\int_{t\wedge\tau_n}^{\tau_n}f^{\prime}_s{\rm d}s\right)^{p}\bigg|\F_{r\wedge\tau_m}\right].
\end{split}
\end{align}
Then, it follows from \eqref{6.10} and \eqref{6.11} that there exists a constant $C_{p,\theta} >0$ depending only on $p$ and $\theta$ such that for each $0\leq r\leq t< +\infty$ and $n\geq m\geq1$,
\begin{align}\label{6.12}
\E\left[\sup_{s\in[t\wedge\tau_n,\tau_n]}|y_s^{\prime}|^{p}\bigg|
\F_{r\wedge\tau_m}\right]\leq
C_{p,\theta}\E\left[|y^{\prime}_{\tau_n}|^{p}+\left(
\int_{t\wedge\tau_n}^{\tau_n}f_s^{\prime}{\rm d}s\right)^{p}\bigg|\F_{r\wedge\tau_m}\right].
\end{align}
In light of $(y_t^{\prime})_{t\in[0,\tau]}\in S_\tau^{p}(0;\R)$ and assumption \ref{A:A}, letting $n\rightarrow \infty$ and using Fatou's lemma on both sides of \eqref{6.12} yields that for each $0\leq r\leq t<+\infty$ and $m\geq1$,
\begin{align*}
\E\left[\sup_{s\in[t\wedge\tau,\tau]}|y_s^{\prime}|^{p}\bigg|
\F_{r\wedge\tau_m}\right]\leq
C_{p,\theta}\E\left[|\xi^{\prime}|^{p}+\left(
\int_{t\wedge\tau}^{\tau}f_s^{\prime}{\rm d}s\right)^{p}\bigg|\F_{r\wedge\tau_m}\right].
\end{align*}
Thus, the desired assertion \eqref{6.1} follows by sending $m\rightarrow\infty$ and using the martingale convergence theorem (see Corollary A.9 in Appendix C of \citet{Oksendal2005}) on both sides of the last inequality.
\end{proof}
Next, we present the other a priori estimate which has already been proved in \cite[Proposition 2.2]{Zhang2026}. We omit its proof here.

\begin{pro}\label{pro:2.3}
Let the assumptions of \cref{pro:2.2} hold. Assume further that
$$\bar\rho_t\geq \bar\mu_t+\frac{\theta}{2[1\wedge(p-1)]}\bar\nu_t^2,\ \ t\in[0,\tau],$$
then $z_\cdot\in M_\tau^{p}(\bar\rho_\cdot;\R^{d})$, and there exists a constant $C_{p,\theta}^{\prime}>0$ depending only on $p$ and $\theta$ such that for each $0\leq r\leq t< +\infty$, we have
$$
\E\left[\left(\int_{t\wedge\tau}^{\tau}e^{2\int_{0}^{s}\bar{\rho}_{r}{\rm d}r}|z_s|^2{\rm d}s\right)^{\frac{p}{2}}\bigg|\F_{r\wedge\tau}\right]
\leq
C_{p,\theta}^{\prime}\E\left[\sup_{s\in[t\wedge\tau,\tau]}\left(e^{p \int_{0}^{s}\bar{\rho}_r{\rm d}r}|y_s|^{p}\right) +\left(\int_{t\wedge\tau}^{\tau}e^{ \int_{0}^{s}\bar{\rho}_r{\rm d}r}\bar{f}_s{\rm d}s\right)^{p}\bigg|\F_{r\wedge\tau}\right].\vspace{0.2cm}
$$
\end{pro}

\begin{rmk}\label{rmk:2.1}
From the proof of \cite[Proposition 2.2]{Zhang2026}, it is not difficult to see that the conclusion of \cref{pro:2.3} does not necessarily hold under the assumptions of \cref{pro:2.2}.
\end{rmk}

In the rest of this paper, we always assume that $p>1$ and $\theta>1$ are two given constants, and $\rho_t(\omega)$, $\mu_t(\omega)$: $\Omega \times [0,\tau]\mapsto \R$ and $\nu_t(\omega)$, $f_t(\omega)$: $\Omega \times [0,\tau]\mapsto \R_+$ are four $(\F_t)$-progressively measurable processes satisfying $\int_{0}^{\tau}\left(|\rho_t|+|\mu_t|+\nu_t^2\right){\rm d}t<+\infty$,
\begin{align}\label{eq:2.1}
\rho_t\geq \mu_t+\frac{\theta}{2(p-1)}\nu_t^2, \ \ t\in[0,\tau]
\end{align}
and
\begin{align}\label{eq:2.2}
\E\left[\left(\int_0^\tau e^{\int_{0}^{t}\rho_r{\rm d}r} f_t{\rm d}t\right)^{p}\right]<+\infty.
\end{align}

\section{An existence result}
\label{sec:3-An existence result}
\setcounter{equation}{0}

In this section, in a more general weighted space than those in \citet{Li2025}, \citet{LiFan2024SD} and \citet{Zhang2026}, we will propose and prove an existence result on the weighted $L^p$ solution of BSDE \eqref{BSDE1.1} with a generator $g$ satisfying a stochastic one-sided linear growth in the state variables $(y,z)$.

\subsection{Statement of the main result}

Let us first introduce the following assumptions on $g$.
\begin{enumerate}
{\addtolength{\leftskip}{2em}\item[\textbf{(H1)} \ (i)]\makeatletter\def\@currentlabel{(H1)}\makeatother\label{A:H1} For each $t\in [0,\tau]$, $g(\omega,t,\cdot,\cdot)$ is continuous.

\item[(ii)] $g$ has a stochastic one-sided linear growth in $(y,z)$, i.e., for each $(t,y,z) \in [0,\tau]\times\R\times\R^d$,
    $$\sgn(y)g(\omega,t,y,z) \leq f_t(\omega)+\mu_t(\omega)|y|+ \nu_t(\omega)|z|.$$

\item[(iii)] $g$ has a general growth in $y$ and a quadratic growth in $z$, i.e., there exist an increasing function $\varphi(x):\R_+\mapsto \R_+$ and a function $\phi_t(\omega,x)\in\mathbf{S}$ such that for each $(t,y,z) \in  [0,\tau]\times\R\times\R^d$,
    $$|g(\omega,t,y,z)| \leq \phi_t(\omega,|y|) + \varphi(|y|)|z|^2.$$}
\end{enumerate}

\begin{thm}\label{thm:3.2}
Let the generator $g$ satisfy assumption \ref{A:H1}. Assume further that either $\mu_\cdot\in \R_+$ or \ref{A:H1}(iii) is replaced with the following stronger condition: $g$ has a general growth in $y$ and a stochastic linear growth in $z$, i.e., there exists a function $\phi_t(\omega,x)\in\mathbf{S}$ such that for each $(t,y,z)\in[0,\tau]\times\R\times\R^d$,
\begin{align}\label{eq:2.3*}
|g(\omega,t,y,z)|\leq \phi_t(\omega,|y|)+\nu_t(\omega)|z|.
\end{align}
Then, for each $\xi\in L_\tau^p(\rho_\cdot;\R)$, BSDE \eqref{BSDE1.1} admits a solution $(y_t,z_t)_{t\in[0,\tau]}$ such that $y_\cdot\in S_\tau^p(\rho_\cdot;\R)$. Furthermore, if the process $\rho_\cdot$ also satisfies
\begin{align}\label{eq:2.3}
\rho_t\geq \mu_t+\frac{\theta}{2[1\wedge(p-1)]}\nu_t^2, \ \ t\in[0,\tau],
\end{align}
then $z_\cdot\in M_\tau^p(\rho_\cdot;\R^{d})$.\vspace{0.1cm}
\end{thm}

\begin{rmk}\label{rmk:3.3}
Note that under the stronger assumptions \eqref{eq:1.1} and \eqref{eq:1.2} of the generator $g$ than \ref{A:H1}(ii), \citet{Li2025} and \citet{LiFan2024SD} only studied the case that $\mu_\cdot\in \R_+$, while \citet{Zhang2026} considered the more general case that $\mu_\cdot\in \R$. As demonstrated in \cite{Zhang2026}, the random variable $\xi$ belonging to $L^p_{\tau}(\rho;\R)$ can be non-integrable when both $\mu_\cdot$ and $\rho_\cdot$ are unbounded negative processes.\vspace{0.1cm}
\end{rmk}

\begin{ex}\label{ex:3.1}
We provide several examples to demonstrate that \cref{thm:3.2} cannot be encompassed by any existing conclusions.
\begin{itemize}
\item [(i)] Let $\tau$ be a finite stopping time, i.e., $\mathbb{P}(\tau<+\infty)=1$, and for each $(t, y,z)\in [0,\tau]\times \R\times\R^{d}$, let
$$
g(t,y,z):=|z|^2 (1-e^y)+|B_t||z|\sin|z|.
$$
Obviously, this $g$ satisfies \ref{A:H1} with
$$
f_t=\mu_t\equiv0,\ \ \nu_t:=|B_t|,\ \ \phi_t(x):\equiv |B_t|^2\ \ {\rm and}\ \ \varphi(x):=e^{x}+2.
$$
Then, it follows from \cref{thm:3.2} that for each $\xi\in L_\tau^p(\rho_\cdot;\R)$, BSDE \eqref{BSDE1.1} admits a solution $(y_t,z_t)_{t\in[0,\tau]}$ such that $y_\cdot\in S_\tau^p(\rho_\cdot;\R)$. Furthermore, if \eqref{eq:2.3} holds, then $z_\cdot\in M_\tau^p(\rho_\cdot;\R^{d})$.

\item [(ii)] Let $\tau$ be a stopping time taking values in the extended nonnegative real numbers and $\sigma>0$ be a constant. For each $(t, y,z)\in [0,\tau]\times \R\times\R^{d}$, let
$$g(t,y,z):=e^{-\int_0^{t\wedge\sigma} |B_s| {\rm d}s-t}-y^3 |z|^2 +|B_t|{\bf 1}_{0\leq t\leq \sigma}|y|.$$
It is evident that this $g$ satisfies \ref{A:H1} with
$$
f_t:=e^{-\int_0^{t\wedge\sigma} |B_s| {\rm d}s-t},\ \ \mu_t:=|B_t|{\bf 1}_{0\leq t\leq \sigma}, \ \ \nu_t\equiv0, \ \ \phi_t(x):=|B_t|{\bf 1}_{0\leq t\leq \sigma}x+e^{-t}\ \ {\rm and} \ \ \varphi(x):=x^3,
$$
and $\rho_t:=\mu_t$ satisfying \eqref{eq:2.3}. Then, it follows from \cref{thm:3.2} that for each $\xi\in L_\tau^p(\rho_\cdot;\R)$, BSDE \eqref{BSDE1.1} admits a weighted $L^p$ solution $(y_t,z_t)_{t\in [0,\tau]}\in H_\tau^p(\rho_\cdot;\R\times\R^{d})$.

\item [(iii)] Let $\tau$ be a bounded stopping time, i.e., $\tau\leq T$ for some constant $T>0$, and for each $(t,y,z)\in [0,\tau]\times\R\times\R^{d}$, let
$$
g(t,y,z):=-y^3 e^{|B_t|^2}-2|B_t|^3y+e^{\int_0^t |B_s|^3{\rm d}s}|B_t|^3\sin(y^2|z|).
$$
It is not hard to verify that this $g$ satisfies \ref{A:H1}(i)-(ii) and \eqref{eq:2.3*} with
$$
f_t:=e^{\int_0^t |B_s|^3{\rm d}s}|B_t|^3,\ \ \mu_t:=-2|B_t|^3,\ \ \nu_t\equiv0\ \ {\rm and}\ \ \phi_t(x):=e^{\int_0^t |B_s|^3{\rm d}s}|B_t|^3+e^{|B_t|^2}x^3+2|B_t|^3 x,
$$
and $\rho_t:=\mu_t$ satisfying \eqref{eq:2.3}. Then, it follows from \cref{thm:3.2} that for each $\xi\in L_\tau^2(\rho_\cdot;\R)$, for example $\xi:=e^{2\int_0^T |B_t|^3 {\rm d}t}$, BSDE \eqref{BSDE1.1} admits a weighted $L^2$ solution $(y_t,z_t)_{t\in [0,\tau]}\in H_\tau^2(\rho_\cdot;\R\times\R^{d})$.

\item [(iv)] For each $(t, y,z)\in [0,\tau]\times \R\times\R^{d}$, let
$$g(t,y,z):=B_t{\bf 1}_{0\leq t\leq 1}y+|B_t|{\bf 1}_{0\leq t\leq 2}|z|\cos y.$$
Obviously, this $g$ satisfies \ref{A:H1}(i)-(ii) and \eqref{eq:2.3*} with
$$
f_t\equiv0,\ \ \mu_t:=B_t{\bf 1}_{0\leq t\leq 1},\ \ \nu_t:=|B_t|{\bf 1}_{0\leq t\leq 2}\ \ {\rm and} \ \ \phi_t(x):=|B_t|{\bf 1}_{0\leq t\leq 1} x.
$$
Then, it follows from \cref{thm:3.2} that for each $\xi\in L_\tau^p(\rho_\cdot;\R)$, BSDE \eqref{BSDE1.1} admits a solution $(y_t,z_t)_{t\in[0,\tau]}$ such that $y_\cdot\in S_\tau^p(\rho_\cdot;\R)$. Furthermore, if \eqref{eq:2.3} holds, then $z_\cdot\in M_\tau^p(\rho_\cdot;\R^{d})$.\vspace{0.2cm}
\end{itemize}
\end{ex}

Let us furhter introduce the following assumptions on the generator $g$.
\begin{enumerate}
{\addtolength{\leftskip}{2em}
\item[\textbf{(H1')} \ (i)]\makeatletter\def\@currentlabel{(H1')}\makeatother\label{A:H1'} $\as$, $g(\omega,t,\cdot,\cdot)$ is continuous.

\item[(ii)] $g$ has a stochastic linear growth in $(y,z)$, i.e., and for each $(t,y,z) \in [0,\tau]\times\R\times\R^d$,
    \begin{equation}\label{eq:3.3*}
    |g(\omega,t,y,z)| \leq f_t(\omega)+\mu_t(\omega)|y|+ \nu_t(\omega)|z|.
    \end{equation}
}
\end{enumerate}

We note that \eqref{eq:3.3*} implies $\mu_\cdot\in\R_+$. Otherwise, a contradiction will arise when we let $z\equiv 0$ and send $y$ to infinity. It is obvious that \ref{A:H1'} is stronger than \ref{A:H1} plus $\mu_\cdot\in\R_+$ as well as \ref{A:H1}(i)-(ii) plus \eqref{eq:2.3*}. Consequently, the following corollary follows immediately from \cref{thm:3.2}. We omit its proof here.

\begin{cor}\label{cor:3.2}
Let the generator $g$ satisfy assumption \ref{A:H1'}. Then, for each $\xi\in L_\tau^p(\rho_\cdot;\R)$, BSDE \eqref{BSDE1.1} admits a solution $(y_t,z_t)_{t\in[0,\tau]}$ such that $y_\cdot\in S_\tau^p(\rho_\cdot;\R)$. Furthermore, if \eqref{eq:2.3} holds, then $z_\cdot\in M_\tau^p(\rho_\cdot;\R^{d})$.\vspace{0.1cm}
\end{cor}

The following corollary can be derived from \cref{thm:3.2} and \cref{cor:3.2} and it demonstrates the existence of a usual $L^p$ solution of BSDE \eqref{BSDE1.1} with unbounded stochastic coefficients.

\begin{cor}\label{cor:3.1}
Let the generator $g$ satisfy assumption \ref{A:H1} plus $\mu_\cdot\in \R_+$ (or assumption \ref{A:H1}(i)-(ii) plus \eqref{eq:2.3*} or assumption \ref{A:H1'}) with the nonnegative process $f_\cdot$ satisfying $\E\left[(\int_0^\tau f_t{\rm d}t)^{p}\right]<+\infty$ instead of \eqref{eq:2.2}. Assume further that for some constant $M>0$,
\begin{align}\label{eq:2.4}
\int_0^\tau\left(\mu_t+\frac{\theta}{2(p-1)}\nu_t^2\right)^+{\rm d}t\leq M.
\end{align}
Then, for each $\xi\in L_\tau^p(0;\R)$, BSDE \eqref{BSDE1.1} admits an adapted solution $(y_t,z_t)_{t\in[0,\tau]}$ such that $y_\cdot\in S_\tau^p(0;\R)$. Furthermore, if
\begin{align}\label{eq:2.4*}
\int_0^\tau\left(\mu_t+\frac{\theta}{2[1\wedge(p-1)]}\nu_t^2\right)^+{\rm d}t\leq M,
\end{align}
then $z_\cdot\in M_\tau^p(0;\R^{d})$.
\end{cor}

\begin{proof}[\bf Proof]
Let $\rho_\cdot:=\mu_\cdot+\frac{\theta}{2 (p-1)}\nu_\cdot^2$ or $\rho_\cdot:=\mu_\cdot+\frac{\theta}{2[1\wedge(p-1)]}\nu_\cdot^2$. It then follows from \eqref{eq:2.4} or \eqref{eq:2.4*} that
$\int_0^\tau \rho^+_t {\rm d}t\leq M$. Furthermore, it is easy to see that $\xi\in L_\tau^p(\rho_\cdot^+;\R)$ and \eqref{eq:2.2} holds with $\rho^+_\cdot$ instead of $\rho_\cdot$, and that $S_\tau^p(\rho_\cdot^+;\R)=S_\tau^p(0;\R)$ and $M_\tau^p(\rho_\cdot^+;\R^{d})=M_\tau^p(0;\R^{d})$. Thus, the desired assertion follows from \cref{thm:3.2} and \cref{cor:3.2}.
\end{proof}

\begin{rmk}\label{rmk:3.1}
Note that for the case of $\mu_\cdot\in \R_+$, both \eqref{eq:2.4} and \eqref{eq:2.4*} are equivalent to
$$
\int_0^\tau\left(\mu_t+\nu_t^2\right){\rm d}t\leq M.
$$
\cref{cor:3.1} strengthens \citet[Theorem 1]{Fan2016SPL}, where the terminal time $\tau$ is a finite positive constant, $\mu_\cdot$ and $\nu_\cdot$ appearing in \ref{A:H1}(ii) and \ref{A:H1'} are two nonnegative constants and the growth condition of $g$ in $z$ (see (H3) therein) is strictly stronger than \ref{A:H1}(iii). By \citet[Remark 3]{Fan2016SPL}, \cref{cor:3.1} also generalizes \citet[Theorem 1]{Lepeltier1997} and \citet[Theorem 4.1]{Briand2007}.
\end{rmk}

\subsection{Proof of Theorem \ref{thm:3.2}}

In this subsection, we give the proof of \cref{thm:3.2} by two steps.\vspace{0.1cm}

{\bf First Step:} We prove that the conclusion of \cref{thm:3.2} holds under the conditions that $\xi\in L_\tau^p(\rho_\cdot;\R)$, the generator $g$ satisfies assumption \ref{A:H1} and $\mu_\cdot\in \R_+$.

For each $n,q \geq 1$ and $(\omega,t,y,z) \in \Omega \times [0,\tau] \times \R \times \R^d$, define $\xi^{n,q}$ and $g^{n,q}(\omega,t,y,z)$ as follows:
\begin{align*}
\begin{split}
\xi^{n,q}:&=\xi^{+} \wedge (ne^{-\int_0^\tau \rho_s {\rm d}s})-\xi^{-} \wedge (qe^{-\int_0^\tau\rho_s {\rm d}s})\\
&= \xi \wedge (ne^{-\int_0^\tau \rho_s {\rm d}s}) \vee (-qe^{-\int_0^\tau \rho_s {\rm d}s})
\end{split}
\end{align*}
and
\begin{align}\label{eq:3.112}
\begin{split}
g^{n,q}(\omega,t,y,z) :&= g^+(\omega,t,y,z) \wedge (ne^{-\int_0^t \rho_s {\rm d}s-t}) - g^-(\omega,t,y,z) \wedge (qe^{-\int_0^t \rho_s {\rm d}s-t})\\
&= g(\omega,t,y,z) \wedge (ne^{-\int_0^t \rho_s {\rm d}s-t}) \vee (-qe^{-\int_0^t \rho_s{\rm d}s-t}).
\end{split}
\end{align}
It is clear that
\begin{align}\label{eq:3.12}
|\xi^{n,q}|\leq |\xi|\wedge [(n+q)e^{-\int_0^\tau \rho_s {\rm d}s}]~~\text{and}~~|g^{n,q}|\leq |g|\wedge [(n+q)e^{-\int_0^t \rho_s {\rm d}s-t}].
\end{align}
Then, it is not very hard to check that the following statements hold true:
\begin{itemize}
\item[(i)] For each $n,q \geq 1$, $\xi^{n,q}$ is bounded by $n+q$ and $g^{n,q}$ is bounded by $(n+q)e^{-t}$;
\item[(ii)] $g^{n,q}$ is increasing in $n$ and decreasing in $q$;
\item[(iii)] All $g^{n,q}$ satisfy assumption \ref{A:H1} with the same parameters: $f_t$, $\rho_t$, $\mu_t$, $\nu_t$, $\phi_t(\cdot)$ and $\varphi(\cdot)$;
\item[(iv)] If $\lim\limits_{q\rightarrow\infty}\lim\limits_{n\rightarrow\infty} y^{n,q}=y$ and $\lim\limits_{q\rightarrow\infty}\lim\limits_{n\rightarrow\infty} z^{n,q}=z$, then $\as$,
$$
\lim\limits_{q\rightarrow\infty}\lim\limits_{n\rightarrow\infty} g^{n,q}(\omega,t,y^{n,q},z^{n,q})=g(\omega,t,y,z).
$$
\end{itemize}
In light of (i), according to in \citet[Lemma 3.4 and Theorem 3.3]{Fan2016SPA}, the following BSDE $(\xi^{n,q}, \tau, g^{n,q})$ admits a minimal bounded solution $(y_{t}^{n,q},z_{t}^{n,q})_{t\in[0,\tau]}$ in the space of $S^\infty_\tau(0;\R)\times M_\tau^2(0;\R^d)$\footnote{It can be verified that Lemma 3.4 and Theorem 3.3 in \citet{Fan2016SPA} remain valid when the constant terminal time $T$ taking values in the extended nonnegative real numbers is replaced with the general random terminal time $\tau$ taking values in the extended nonnegative real numbers.}:
\begin{align}\label{BSDE2.2}
y_t^{n,q}=\xi^{n,q}+\int_t^\tau g^{n,q}(s,y_s^{n,q},z_s^{n,q}){\rm d}s-\int_t^\tau z_s^{n,q}\cdot{\rm d}B_s, \ \ t\in[0,\tau].
\end{align}
And, in light of (ii), $(y_t^{n,q})_{t\in[0,\tau]}$ is nondecreasing in $n$ and non-increasing in $q$. Furthermore, it follows from \eqref{eq:3.12} and \eqref{BSDE2.2} that for each $n,q\geq 1$ and $t\geq 0$,
\begin{align}\label{eq:3.12*}
|y_{t\wedge\tau}^{n,q}|e^{\int_0^{t\wedge\tau}\rho_s{\rm d}s}&=|\E\left[y_{t\wedge\tau}^{n,q}|\F_{t\wedge\tau}\right]|e^{\int_0^{t\wedge\tau}\rho_s{\rm d}s}\leq\E\left[|\xi^{n,q}|+\int_{t\wedge\tau}^\tau |g^{n,q}(s,y_s^{n,q},z_s^{n,q})|{\rm d}s\bigg|\F_{t\wedge\tau}\right]e^{\int_0^{t\wedge\tau}\rho_s{\rm d}s}\nonumber\\
&\leq \E\left[e^{\int_0^\tau \rho_s{\rm d}s}|\xi^{n,q}|\bigg|\F_{t\wedge\tau}\right]+\E\left[\int_{t\wedge\tau}^\tau e^{\int_0^s \rho_r{\rm d}r}|g^{n,q}(s,y_s^{n,q},z_s^{n,q})|{\rm d}s\bigg|\F_{t\wedge\tau}\right]\nonumber\\
&\leq (n+q)+\E\left[\int_0^\tau (n+q)e^{-s}{\rm d}s\bigg|\F_{t\wedge\tau}\right]\leq2(n+q).
\end{align}
Hence, $y_\cdot^{n,q}\in S_\tau^p(\rho_\cdot;\R)$. In light of (iii), we know that $g^{n,q}$ satisfies \ref{A:A} with $\bar{\mu}_t:=\mu_t$, $\bar{\nu}_t:=\nu_t$, $\bar{\rho}_t:=\rho_t$ and $\bar{f}_t:=f_t$. It then follows from \cref{pro:2.2} that there exists a constant $C_{p,\theta}>0$ depending only on $p$ and $\theta$ such that for each $n,q\geq1$ and $t\geq0$,
\begin{align}\label{eq:3.13}
\begin{split}
|y_{t}^{n,q}|^p &\leq \E\left[\sup_{s\in[t\wedge\tau,\tau]}\left(e^{p\int_{0}^{s}\rho_r{\rm d}r}|y_s^{n,q}|^{p}\right)\bigg|\F_{t}\right]\\
&\leq
C_{p,\theta}\E\left[e^{p\int_{0}^{\tau}\rho_r{\rm d}r}|\xi|^{p}+\left(
\int_{0}^{\tau}e^{ \int_{0}^{s}\rho_r{\rm d}r}f_s{\rm d}s\right)^{p}\bigg|\F_{t}\right]=:|M_{t}|^p.
\end{split}
\end{align}

Now, for each pair of integers $m,l \geq 1$, we define the following two stopping times:
$$ \tau_m = \inf\left\{t\geq0 : |M_t| \geq m\right\} \wedge \tau $$
and
$$
\sigma_{m,l} = \inf\left\{t\geq0 : \int_0^t \phi_s(m){\rm d}s \geq l\right\} \wedge \tau_m\vspace{0.1cm}
$$
with the convention that $\inf \emptyset=+\infty$. Then $(y_{m,l}^{n,q}(t), z_{m,l}^{n,q}(t))_{t\in[0,\tau]}:= (y_{t\wedge\sigma_{m,l}}^{n,q}, z_t^{n,q}1_{t \leq \sigma_{m,l}})_{t\in[0,\tau]}$ is a solution in $S^\infty_\tau(0;\R)\times M_\tau^2(0;\R^d)$ to the following BSDE:
$$y_{m,l}^{n,q}(t) = y_{\sigma_{m,l}}^{n,q} + \int_t^\tau 1_{s \leq \sigma_{m,l}} g^{n,q}(s,y_{m,l}^{n,q}(s),z_{m,l}^{n,q}(s)){\rm d}s - \int_t^\tau z_{m,l}^{n,q}(s)\cdot{\rm d}B_s, \ \ t\in[0,\tau].$$
Note that for each fixed $m,l\geq1$, $(y_{m,l}^{n,q}(t))_{t\in[0,\tau]}$ is nondecreasing in $n$ and non-increasing in $q$, and that $\as$, $(g^{n,q})_{n,q}$ converges locally uniformly in $(y,z)$ to $g$ as $n,q\rightarrow \infty$. By \eqref{eq:3.13} along with the definitions of $\tau_m$ and $\sigma_{m,l}$ we can obtain that
$$
\as,\ \ \ \sup_{n,q\geq 1}|y_{m,l}^{n,q}(t)| \leq m.
$$
Furthermore, since $|g^{n,q}|\leq |g|$ for each $n,q\geq1$ and $g$ satisfies assumption \ref{A:H1}(iii), we know that for each $(s,y,z)\in[0,\tau]\times[-m,m]\times\R^d,$
\begin{align*}
\sup_{n,q\geq 1}(1_{s\leq \sigma_{m,l}}|g^{n,q}(s,y,z)|)\leq 1_{s\leq \sigma_{m,l}}\phi_s(m)+\varphi(m)|z|^2.
\end{align*}
By the definition of $\sigma_{m,l}$, we know that $\int_0^\tau 1_{s\leq \sigma_{m,l}}\phi_s(m) {\rm d}s\leq l$. Thus, for each $m,l\geq 1$, we can apply the stability result for the bounded solutions of BSDEs, see \citet[Proposition 3.1]{LuoFan2018}\footnote{It can be verified that Proposition 3.1 in \citet{LuoFan2018} remains valid when the constant terminal time $T$ taking values in the extended nonnegative real numbers is replaced with the random terminal time $\tau$ taking values in the extended nonnegative real numbers.}. Set
$$
y_{m,l}(t):=\inf_{q\geq1}\sup_{n\geq1} y_{t\wedge \sigma_{m,l}}^{n,q}.
$$
Then $(y_{m,l}(t))_{t\in[0,\tau]}$ is continuous and the sequence of processes $(z_t^{n,q}1_{t\leq \sigma_{m,l}})_{t\in[0,\tau]}$ converges to $(z_{m,l}(t))_{t\in[0,\tau]}$ strongly in $M_\tau^2(0;\R^d)$ as $n,q\rightarrow \infty$ such that
$$
y_{m,l}(t)=\inf_{q\geq1}\sup_{n\geq1} y_{\sigma_{m,l}}^{n,q}+ \int_t^\tau 1_{s\leq \sigma_{m,l}} g(s,y_{m,l}(s),z_{m,l}(s)){\rm d}s- \int_t^\tau z_{m,l}(s)\cdot{\rm d}B_s, \ \ t \in [0,\tau].
$$
In light of the last equation and the stability of stopping times $\tau_m$ and $\sigma_{m,l}$, since for each $m,l\geq 1$ and $t\in[0,\tau]$, we have
$$
y_{m+1,l+1}(t\wedge\sigma_{m,l})=y_{m,l+1}(t\wedge\sigma_{m,l})=y_{m,l}(t\wedge\sigma_{m,l})= \inf_{q\geq1}\sup_{n\geq1} y_{t\wedge \sigma_{m,l}}^{n,q}
$$
and
$$
z_{m+1,l+1}(t)1_{t\leq \sigma_{m,l}}=z_{m,l+1}(t)1_{t\leq \sigma_{m,l}}=z_{m,l}(t)1_{t\leq \sigma_{m,l}}= \lim\limits_{n,q\rightarrow\infty}z_t^{n,q}(t)1_{t\leq \sigma_{m,l}},\vspace{0.1cm}
$$
it can be concluded that $(y_t,z_t)_{t\in[0,\tau]}$ is an adapted solution of BSDE$(\xi,\tau,g)$, where
\begin{align*}
y_t:=\inf_{q\geq1}\sup_{n\geq1} y_t^{n,q}~~ \text{and}~~ z_t:= \sum_{m=1}^{+\infty}\left( \sum_{l=1}^{+\infty}z_{m,l}(t)1_{t\in[\sigma_{m,l-1},\sigma_{m,l})}\right)1_{t\in[\tau_{m-1},\tau_m)}, \ t\in[0,\tau]
\end{align*}
with $\tau_0:=1$ and $\sigma_{m,0}:=0$ for each $m\geq 1$. Furthermore, by sending $n,q\rightarrow \infty$ in \eqref{eq:3.13} with $t=0$ and using Fatou's lemma, we can deduce that $(y_t)_{t\in[0,\tau]}\in S_\tau^p(\rho_\cdot;\R)$. Finally, if \eqref{eq:2.3} holds, then the desired assertion that $(z_t)_{t\in[0,\tau]}\in M_\tau^p(\rho_\cdot;\R^d)$ follows from \cref{pro:2.3}.\vspace{0.2cm}

{\bf Second Step:} By using the exponential shift transformation and the conclusion obtained in the first step, we prove that the conclusion of \cref{thm:3.2} holds under the assumption \ref{A:H1} with \eqref{eq:2.3*} instead of \ref{A:H1}(iii) and the condition of $\xi\in L_\tau^p(\rho_\cdot;\R)$. First, we let
\begin{align}\label{eq:3.61}
\tilde{\rho}_\cdot:=\rho_\cdot-\mu_\cdot.
\end{align}
Then, by \eqref{eq:2.1} we know that
$\int_0^\tau \tilde{\rho}_t{\rm d}t<+\infty$ and
\begin{align}\label{eq:3.62}
\tilde{\rho}_t\geq \frac{\theta}{2(p-1)}\nu_t^2, \ \ t\in[0,\tau].
\end{align}
Note that $(y_t,z_t)_{t\in[0,\tau]}$ is an adapted solution of BSDE \eqref{BSDE1.1} in the space of $H_\tau^p(\rho_\cdot;\R\times\R^{d})$ if and only if
$$
(\tilde{y}_t,\tilde{z}_t)_{t\in[0,\tau]}:=(e^{\int_0^t\mu_r{\rm d}r}y_t,e^{\int_0^t\mu_r{\rm d}r}z_t)_{t\in[0,\tau]}.
$$
is an adapted solution of the following BSDE$(\tilde{\xi},\tau,\tilde{g})$ in the space of $H_\tau^p(\tilde{\rho}_\cdot;\R\times\R^{d})$:
\begin{align}\label{BSDE2.3}
\tilde{y}_t=\tilde{\xi}+\int_t^\tau \tilde{g}(s,\tilde{y}_s,\tilde{z}_s){\rm d}s-\int_t^\tau \tilde{z}_t\cdot{\rm d}B_s, \ \ t\in[0,\tau],
\end{align}
where $\tilde{\xi}:=e^{\int_0^\tau\mu_r{\rm d}r}\xi$ and
\begin{align}\label{eq:3.7}
\tilde{g}(t,y,z):=e^{\int_0^t\mu_r{\rm d}r}g(t,e^{-\int_0^t\mu_r{\rm d}r}y,e^{-\int_0^t\mu_r{\rm d}r}z)-\mu_t y,\ \ (t,y,z)\in [0,\tau]\times \R\times \R^d.
\end{align}
Consequently, it suffices to verify that BSDE \eqref{BSDE2.3} admits a solution $(\tilde{y}_t,\tilde{z}_t)_{t\in[0,\tau]}$ such that $\tilde{y}_\cdot\in S_\tau^p(\tilde{\rho}_\cdot;\R)$, and if \eqref{eq:2.3} holds, then $\tilde{z}_\cdot\in M_\tau^p(\tilde{\rho}_\cdot;\R^{d})$. Indeed, since $\xi\in L_\tau^p(\rho_\cdot;\R)$ and $\tilde{\xi}:=e^{\int_0^\tau\mu_r{\rm d}r}\xi$, it follows from \eqref{eq:3.61} that $\tilde{\xi}\in L_\tau^p(\tilde{\rho}_\cdot;\R)$. Since \ref{A:H1}(i) holds for $g$, it is not difficult to verify that the generator $\tilde{g}$ defined in \eqref{eq:3.7} also satisfies \ref{A:H1}(i). Moreover, by virtue of \eqref{eq:3.7}, \ref{A:H1}(ii) and \eqref{eq:2.3*}, we deduce that for each $(t,y,z) \in[0,\tau]\times\R\times\R^d$,
\begin{align*}
\sgn(y)\tilde{g}(\omega,t,y,z) \leq \tilde{f}_t(\omega)+\nu_t(\omega)|z|\vspace{-0.2cm}
\end{align*}
and\vspace{-0.2cm}
\begin{align*}
|\tilde{g}(\omega,t,y,z)| \leq e^{\int_0^t\mu_r{\rm d}r} \phi_t(e^{-\int_0^t\mu_r{\rm d}r}|y|)+|\mu_t||y|+\nu_t|z| \leq \tilde{\phi}_t(\omega,|y|)+|z|^2,
\end{align*}
where
$$
\tilde{f}_t:=e^{\int_0^t\mu_r{\rm d}r}f_t\ \  {\rm and}\ \  \tilde{\phi}_t(x):=e^{\int_0^t\mu_r{\rm d}r} \phi_t(e^{-\int_0^t\mu_r{\rm d}r}x)+|\mu_t|x+\nu_t^2.\vspace{0.2cm}
$$
Since $\phi_t(x)\in\mathbf{S}$ and $\int_0^\tau (|\mu_t|+\nu_t^2){\rm d}t<+\infty$, it follows from \eqref{eq:3.61} and \eqref{eq:2.2} that $\tilde{\phi}_t(x)\in\mathbf{S}$ and
\begin{align*}
\E\left[\left(\int_0^\tau e^{\int_{0}^{t}\tilde{\rho}_r{\rm d}r}\tilde{f}_t{\rm d}t\right)^p\right]=\E\left[\left(\int_0^\tau e^{\int_{0}^{t}\rho_r{\rm d}r} f_t{\rm d}t\right)^p\right]<+\infty.
\end{align*}
So, the generator $\tilde{g}$ of BSDE \eqref{BSDE2.3} satisfies \ref{A:H1}(ii)-(iii) with $\mu_t\equiv0$ and $\varphi(\cdot)\equiv1$, and with $\tilde{\rho}_t$, $\tilde{\phi}_t(x)$ and $\tilde{f}_t$ instead of $\rho_t$, $\phi_t(x)$ and $f_t$, respectively. Thus, in light of \eqref{eq:3.62}, the desired assertion can be derived from the conclusion obtained the first step. The proof of \cref{thm:3.2} is then completed.

\begin{rmk}\label{rmk:3.2}
Note that the method used to handle the case of $\mu_\cdot\in \R_+$ in the first step can not be applied to the proof in the second step. In fact, since the process $\mu_\cdot$ takes values in $\R$, the truncated function $g^{n,q}$ defined in \eqref{eq:3.112} does not satisfy \ref{A:H1}(ii) in general, and then the uniform estimate with respect to the process $y_\cdot^{n,q}$ in \eqref{eq:3.13} can not be obtained via \cref{pro:2.2}. Instead, we use the exponential shift transformation to transform BSDE \eqref{BSDE1.1} into BSDE \eqref{BSDE2.3}, where the generator $\tilde g$ satisfies \ref{A:H1} with $\mu_\cdot\equiv0$, thereby completing the proof of the second step via the conclusion of the first step.
\end{rmk}

\section{An existence and uniqueness result}
\label{sec:4-An existence and uniqueness result}
\setcounter{equation}{0}

In this section, we will establish an existence and uniqueness theorem and a comparison theorem for the adapted solution of BSDE \eqref{BSDE1.1} in a general weighted $L^p$ space with the same weighting factor $e^{\int_0^t \rho_r {\rm d}r}$ as used in Section 3. The generator $g$ satisfies a stochastic monotonicity condition with a general growth in the state variable $y$ and a stochastic Lipschitz continuity condition in the state variable $z$.

\subsection{Statement of the main results}
First, let us introduce the following assumptions on the generator $g$.
\begin{enumerate}
{\addtolength{\leftskip}{2em}\item[\textbf{(H2)} \ (i)]\makeatletter\def\@currentlabel{(H2)}\makeatother\label{A:H2} $\E\left[\left(\int_0^\tau e^{\int_{0}^{s}\rho_r{\rm d}r}|g(s,0,0)|{\rm d}s\right)^p\right]<+\infty.$

\item[(ii)] For each $(t,z)\in [0,\tau]\times\R^{d}$, $g(\omega,t,\cdot,z)$ is continuous.

\item[(iii)] $g$ has a general growth in $y$, i.e., for each $x\in \R_+$, it holds that
\begin{align}\label{eq:4.0}
\int_0^\tau \psi_t(x){\rm d}t<+\infty
\end{align}
with
\begin{align}\label{eq:4.8}
\psi_t(x):=\sup_{|y|\leq x} \left|g(t,y,0)-g(t,0,0)\right|, \ \ (t,x)\in [0,\tau]\times \R_+.
\end{align}

\item[(iv)] $g$ satisfies a stochastic monotonicity condition in $y$, i.e., for each $(t,y_1,y_2,z)\in [0,\tau]\times\R\times\R\times\R^{d}$,
$$
\sgn(y_1-y_2)(g(\omega,t,y_1,z)-g(\omega,t,y_2,z))\leq \mu_t(\omega)|y_1-y_2|.
$$

\item[(v)] $g$ satisfies a stochastic Lipschitz condition in $z$, i.e.,  for each $(t,y,z_1,z_2)\in [0,\tau]\times\R\times\R^{d}\times\R^{d}$,
$$\left|g(\omega,t,y,z_1)-g(\omega,t,y,z_2)\right|\leq \nu_t(\omega)|z_1-z_2|.$$}
\end{enumerate}

\begin{rmk}\label{rmk:4.1}
Note that \ref{A:H2}(ii) and \ref{A:H2}(iv)-(v) come from \citet{Zhang2026}, and \ref{A:H2}(i) generalizes the corresponding assumption in \citet{Zhang2026} (see (H1) therein) by broadening the scope of the process $\rho_\cdot$ to be only bigger than $\mu_\cdot +\frac{\theta}{2(p-1)}\nu_\cdot^2$. In addition, \ref{A:H2}(iii) can be compared with the general growth condition used in \citet{Zhang2026} (see (H3) therein), which is presented as follows: there exists an $(\F_t)$-progressively measurable real-valued process $(\alpha_t)_{t\in[0,\tau]}$ satisfying $\essinf\limits_{t\in[0,\tau]}\alpha_t>0$ such that
\begin{align}\label{eq:4.6}
\forall x\in\R_+, \ \ \ \E\left[\int_0^\tau \alpha_t\psi_t(x\alpha_t){\rm d}t\right]<+\infty,
\end{align}
where the function $\psi_t(\cdot)$ is defined in \eqref{eq:4.8}. Obviously, \eqref{eq:4.6} can imply that $\mathbb{P}-a.s.,$ for each $x\in\R_+$,
\begin{align}\label{eq:4.7}
\int_0^\tau \alpha_t\psi_t(x\alpha_t){\rm d}t<+\infty.
\end{align}
Furthermore, it is not difficult to verify that \eqref{eq:4.7} is equivalent to \eqref{eq:4.0}. In fact, it is evident that if the function $\psi_t(\cdot)$ satisfies \eqref{eq:4.0}, then it also satisfies \eqref{eq:4.7} by taking $\alpha_t\equiv1$. On the other hand, if the function $\psi_t(\cdot)$ satisfies \eqref{eq:4.7}, in light of $\essinf\limits_{t\in[0,\tau]}\alpha_t>0$, we deduce that $\mathbb{P}-a.s.,$ for each $x\in\R_+$,
\begin{align*}
\int_0^\tau \psi_t(x){\rm d}t=\int_0^\tau \frac{1}{\alpha_t}\alpha_t\psi_t(x \frac{1}{\alpha_t}\alpha_t){\rm d}t\leq \frac{1}{\essinf\limits_{t\in[0,\tau]}\alpha_t} \int_0^\tau \alpha_t\psi_t(x'\alpha_t){\rm d}t<+\infty
\end{align*}
with
$$x':=\frac{x}{\essinf\limits_{t\in[0,\tau]}\alpha_t}\in\R_+,$$
and then \eqref{eq:4.0} holds for the function $\psi_t(\cdot)$. Therefore, \ref{A:H2}(iii) is weaker than \eqref{eq:4.6}. However, up to now we do not know whether \ref{A:H2}(iii) is strictly weaker than \eqref{eq:4.6}.
\end{rmk}

The following theorem establishes an existence and uniqueness result for the adapted solution of BSDE \eqref{BSDE1.1} in a general weighted $L^p$ space, which is the first main result in this section.

\begin{thm}\label{thm:4.2}
Let the generator $g$ satisfy \ref{A:H2}. Then, for each $\xi\in L_\tau^p(\rho_\cdot;\R)$, BSDE \eqref{BSDE1.1} admits a unique solution $(y_t,z_t)_{t\in[0,\tau]}$ such that $y_\cdot\in S_\tau^p(\rho_\cdot;\R)$. Furthermore, if \eqref{eq:2.3} holds, then $z_\cdot\in M_\tau^p(\rho_\cdot;\R^{d})$.
\end{thm}

\begin{rmk}\label{rmk:2.2}
Regarding \cref{thm:4.2}, we make the following several remarks.
\begin{itemize}
\item [(i)] Based on \cref{rmk:4.1}, we know that \cref{thm:4.2} improves the one-dimensional case of Theorem 3.2 in \citet{Zhang2026}, where \ref{A:H2}(iii) is replaced with the stronger condition \eqref{eq:4.6}.

\item [(ii)] For the case of $\mu_\cdot+\frac{\theta}{2(p-1)}\nu_\cdot^2\leq\rho_\cdot< \mu_\cdot+\frac{\theta}{2[1\wedge(p-1)]}\nu_\cdot^2$, \cref{thm:4.2} is established for the first time and it further strengthens those existing studies on BSDEs with stochastic monotonicity generators, including the one-dimensional case of \cite[Theorem 3.2]{Zhang2026}, \cite[Theorem 3.1]{LiFan2024SD} and
    \cite[Theorem 3.1]{Li2025}.
\item [(iii)] Using an identical analysis to \cref{cor:3.1}, by virtue of \cref{thm:4.2} we can derive an existence and uniqueness result of a usual $L^p$ solution of BSDE \eqref{BSDE1.1} with unbounded stochastic coefficients under assumption \ref{A:H2} provided that \eqref{eq:2.4} and \eqref{eq:2.4*} are in force and $\E\left[\left(\int_0^\tau |g(s,0,0)|{\rm d}s\right)^p\right]<+\infty$.
\end{itemize}
\end{rmk}

\begin{ex}\label{ex:4.1}
We would like to present two examples where \cref{thm:4.2} is applicable, but \citet[Theorem 3.2]{Zhang2026} and any existing results are not applicable. In the following two examples, we always assume that $p>2$, $1<\theta\leq p-1$ and
\begin{align}\label{eq:4.5}
\rho_\cdot:=\mu_\cdot+\frac{\theta}{2(p-1)}\nu_\cdot^2,
\end{align}
which satisfies \eqref{eq:2.1}. However, since $p>2$ and $1<\theta\leq p-1$, this $\rho_\cdot$ does not satisfy \eqref{eq:2.3}.
\begin{itemize}
\item [(i)] Let $\sigma>0$ be a constant. For each $(t, y,z)\in [0,\tau]\times \R\times\R^{d}$, let
$$g(t,y,z):=e^{-|B_t|^3{\bf 1}_{0\leq t\leq \sigma}y}+(|B_t|^2{\bf 1}_{0\leq t\leq \sigma}+e^{-t})|z|-1.$$
It is easy to verify that this generator $g$ satisfies assumption \ref{A:H2} with
$$
g(t,0,0)\equiv0, \ \ \mu_t\equiv 0 \ \ \text{and} \ \ \nu_t:=|B_t|^2{\bf 1}_{0\leq t\leq \sigma}+e^{-t}, \ \ t\in[0,\tau],
$$
and with $\rho_t:=\frac{\theta}{2(p-1)}(|B_t|^4{\bf 1}_{0\leq t\leq \sigma}+e^{-t})$ satisfying \eqref{eq:2.1} but not \eqref{eq:2.3}.
Then, it follows from \cref{thm:4.2} that for each $\xi\in L_\tau^p(\rho_\cdot;\R)$, for example $\xi:=e^{-\frac{\theta}{2(p-1)}\int_0^\sigma |B_t|^4 {\rm d}t}$, BSDE \eqref{BSDE1.1} admits a unique solution $(y_t,z_t)_{t\in[0,\tau]}$ such that $(y_t)_{t\in[0,\tau]}\in S_\tau^p(\rho_\cdot;\R)$.

\item [(ii)] Let $p=3$ and $\theta=2$, and for each $(t, y,z)\in [0,\tau]\times \R\times\R^{d}$, let
$$g(t,y,z):=B_t{\bf 1}_{0\leq t\leq 1}(y+\sin|z|).$$
Obviously, this generator $g$ satisfies assumption \ref{A:H2} with
$$
g(t,0,0)\equiv0, \ \ \mu_t:=B_t{\bf 1}_{0\leq t\leq 1}, \ \ \nu_t:=|B_t|{\bf 1}_{0\leq t\leq 1}, \ \ t\in[0,\tau],
$$
and $\rho_t:=(B_t+\frac{1}{2}|B_t|^2){\bf 1}_{0\leq t\leq 1}$ satisfying \eqref{eq:2.1} but not \eqref{eq:2.3}.
Then, by \cref{thm:4.2} we know that for each $\xi\in L_\tau^3(\rho_\cdot;\R)$, BSDE \eqref{BSDE1.1} admits a unique solution $(y_t,z_t)_{t\in[0,\tau]}$ such that $y_\cdot\in S_\tau^3(\rho_\cdot;\R)$.\vspace{0.2cm}
\end{itemize}
\end{ex}

The following comparison theorem is the second main result of this section. It improves Theorem 5.1 in \citet{LiFan2024SD} where the weighting factor is $e^{\int_0^t (\beta\mu_s^+ +\frac{\theta}{2[1\wedge(p-1)]} \nu_s^2){\rm d}s}$ for some constant $\beta\geq1$, and the uniqueness part of \cref{thm:4.2} is a direct consequence of this comparison theorem.

\begin{thm}\label{thm:4.1}
Assume that $\xi$ and $\xi'$ are two terminal values, $g$ and $g^{\prime}$ are two generators, and $(Y_\cdot, Z_\cdot)$ and $\left(Y_\cdot^{\prime}, Z_\cdot^{\prime}\right)$ are, respectively, solutions of BSDE$(\xi,\tau,g)$ and BSDE$\left(\xi^{\prime}, \tau, g^{\prime}\right)$. If $(Y_\cdot-Y_\cdot^{\prime})^+\in S_\tau^p(\rho_\cdot;\R)$ and $\xi \leq \xi^{\prime}$, and either of the following conditions is satisfied:\vspace{0.2cm}

(i) $g$ satisfies \ref{A:H2}(iv)-(v), and $g\left(t, Y_{t}^{\prime}, Z_{t}^{\prime}\right) \leq g^{\prime}\left(t, Y_{t}^{\prime}, Z_{t}^{\prime}\right), \ \ t\in[0,\tau];$\vspace{0.1cm}

(ii) $g^{\prime}$ satisfies \ref{A:H2}(iv)-(v), and $g\left(t, Y_{t}, Z_{t}\right) \leq g^{\prime}\left(t, Y_{t}, Z_{t}\right), \ \ t\in[0,\tau],$\vspace{0.2cm}

\noindent then for each $t \in[0,\tau]$, we have $Y_{t} \leq Y_{t}^{\prime}.$
\end{thm}
\vspace{0.1cm}

\subsection{Proof of Theorems \ref{thm:4.2} and \ref{thm:4.1}}\vspace{0.1cm}

In this subsection, we will give proofs of \cref{thm:4.2} and \cref{thm:4.1}. First, we prove \cref{thm:4.1}. The proof is partially motivated by \cite[Theorem 5.1]{LiFan2024SD}. However, some technical adjustments have to be made to address those difficulties caused by the general case that $\mu_\cdot$ and $\rho_\cdot$ in \eqref{eq:2.1} take values in $\R$.

\begin{proof}[\bf Proof of \cref{thm:4.1}]
We only prove the case (i), and the case (ii) can be proved similarly. For each $n\geq1$, denote the following $(\F_t)$-stopping time:
$$\tau_{n}:=\inf \left\{t\geq0: \int_{0}^{t}\nu_{s}^{2}{\rm d}s+ \int_{0}^{t}e^{2\int_{0}^{s}\mu_r{\rm d}r}\left(|Z_s|^2+|Z_s^{\prime}|^2\right){\rm d}s \geq n\right\} \wedge \tau,$$
with the convention that $\inf \emptyset=+\infty$.
Let
$$U_\cdot:=Y_\cdot-Y_\cdot', \ \ V_\cdot:=Z_\cdot-Z_\cdot', \ \ \zeta:=\xi-\xi'.$$
Then $U_\cdot^+\in S_\tau^p(\rho_\cdot;\R)$ and
$$
U_{t}=\zeta+\int_{t}^{\tau} \left(g(s, Y_{s}, Z_{s})-g^{\prime}(s, Y_{s}^{\prime}, Z_{s}^{\prime})\right){\rm d}s-\int_{t}^{\tau} V_{s}\cdot{\rm d}B_{s}, \ \ t\in[0,\tau].
$$
By applying It\^{o}-Tanaka's formula to $U^+_{t}e^{\int_{0}^{t}\mu_s{\rm d}s}$ in $[t\wedge\tau_n,\tau_n]$, we obtain that for each $n\geq1$ and $t\geq0$,
\begin{align}\label{eq:4.1}
\begin{split}
U^+_{t\wedge\tau_n}e^{\int_{0}^{t\wedge\tau_n}\mu_s{\rm d}s}
\leq& U^+_{\tau_n}e^{\int_{0}^{\tau_n}\mu_s{\rm d}s}-\int_{t\wedge\tau_n}^{\tau_n}e^{\int_{0}^{s}\mu_r{\rm d}r}{\bf 1}_{U_s>0}V_s\cdot{\rm d}B_s\\
&+\int_{t\wedge\tau_n}^{\tau_n}e^{\int_{0}^{s}\mu_r{\rm d}r}\left[{\bf 1}_{U_s>0}\left(g(s, Y_{s}, Z_{s})-g'(s, Y'_{s}, Z'_{s})\right)-\mu_s U^+_s\right]{\rm d}s.
\end{split}
\end{align}
Since $g(s, Y'_{s}, Z'_{s})-g'(s, Y'_{s}, Z'_{s})$ is non-positive, combining assumptions \ref{A:H2}(iv)-(v) we deduce that
\begin{align}\label{eq:4.2}
\begin{split}
&{\bf 1}_{U_s>0}\left(g(s, Y_{s}, Z_{s})-g'(s, Y'_{s}, Z'_{s})\right)-\mu_s U^+_s\\
&\ \  \ = {\bf 1}_{U_s>0}\left(g(s, Y_{s}, Z_{s})-g(s, Y'_{s}, Z'_{s})+g(s, Y'_{s}, Z'_{s})-g'(s, Y'_{s}, Z'_{s})\right)-\mu_sU^+_s\\
&\ \  \ \leq {\bf 1}_{U_s>0}\left(\mu_s|U_s|+\nu_s|V_s|\right)-\mu_sU^+_s={\bf 1}_{U_s>0}\nu_s|V_s|, \ \ s\in[0,\tau].
\end{split}
\end{align}
Denote
$$b_s:=\frac{\nu_sV_s}{|V_s|}{\bf 1}_{|V_s|\neq0}, \ \ s\in[0,\tau].$$
Then, in light of \eqref{eq:4.1} and \eqref{eq:4.2}, we can deduce that for each $n\geq1$,
\begin{align}\label{eq:4.3}
\begin{split}
U^+_{t\wedge\tau_n}e^{\int_{0}^{t\wedge\tau_n}\mu_s{\rm d}s}
&\leq U^+_{\tau_n}e^{\int_{0}^{\tau_n}\mu_s{\rm d}s}-\int_{t\wedge\tau_n}^{\tau_n}e^{\int_{0}^{s}\mu_r{\rm d}r}{\bf 1}_{U_s>0}V_s \cdot{\rm d}B_s+\int_{t\wedge\tau_n}^{\tau_n}e^{\int_{0}^{s}\mu_r{\rm d}r}{\bf 1}_{U_s>0}V_s \cdot b_s{\rm d}s\\
&= U^+_{\tau_n}e^{\int_{0}^{\tau_n}\mu_s{\rm d}s}-\int_{t\wedge\tau_n}^{\tau_n}e^{\int_{0}^{s}\mu_r{\rm d}r}{\bf 1}_{U_s>0}V_s\cdot \left(-b_s{\rm d}s+{\rm d}B_s\right), \ \ t\geq0.
\end{split}
\end{align}
By the definition of $\tau_n$ we know that the following Novikov's condition is fulfilled:
$$
\E\left[{\rm exp}\left(\frac{1}{2}\int_0^\tau{\bf 1}_{s\leq\tau_n}|b_s|^2{\rm d}s\right)\right]<+\infty.
$$
Then, let $\mathbb{Q}_n$ be the probability measure on $(\Omega,\F_\tau)$ which is equivalent to $\mathbb{P}$ and defined by
\begin{align*}
\frac{{\rm d}\mathbb{Q}_n}{{\rm d}\mathbb{P}}:={\rm exp}\left\{\int_0^\tau{\bf 1}_{s\leq\tau_n}b_s\cdot{\rm d}B_s-\frac{1}{2}\int_0^\tau{\bf 1}_{s\leq\tau_n}|b_s|^2{\rm d}s\right\}.
\end{align*}
According to Girsanov's theorem, by taking the conditional mathematical expectation with respect to $\F_t$ under $\mathbb{Q}_n$ on both sides of \eqref{eq:4.3} we can deduce that for each $n\geq1$,
\begin{align*}
\begin{split}
U^+_{t\wedge\tau_n}e^{\int_{0}^{t\wedge\tau_n}\mu_s{\rm d}s}\leq \E_{\mathbb{Q}_n}\left[U^+_{\tau_n}e^{\int_{0}^{\tau_n}\mu_s{\rm d}s}\bigg|\F_t\right]=\frac{\E\left[U^+_{\tau_n}e^{\int_{0}^{\tau_n}\mu_s{\rm d}s}\frac{{\rm d}\mathbb{Q}_n}{{\rm d}\mathbb{P}}\bigg|\F_t\right]}{\E\left[\frac{{\rm d}\mathbb{Q}_n}{{\rm d}\mathbb{P}}\bigg|\F_t\right]}, \ \ t\in[0,\tau].
\end{split}
\end{align*}
Then we have
\begin{align}\label{eq:4.4}
\begin{split}
U^+_{t\wedge\tau_n}e^{\int_{0}^{t\wedge\tau_n}\mu_s{\rm d}s}\leq\E\left[\Delta_t^n\bigg|\F_t\right], \ \ t\in[0,\tau],
\end{split}
\end{align}
where
$$
\Delta_t^n:=U^+_{\tau_n}e^{\int_{0}^{\tau_n}\mu_s{\rm d}s+\int_{t\wedge\tau_n}^{\tau_n}b_s\cdot{\rm d}B_s-\frac{1}{2}\int_{t\wedge\tau_n}^{\tau_n}|b_s|^2{\rm d}s}, \ \ t\in[0,\tau].\vspace{0.1cm}
$$
Next, we can verify that $(\Delta_t^n)^\infty_{n=1}$ is uniformly integrable for each $t\in[0,\tau]$.
Indeed, in light of $p>1$ and $\theta>1$, there exists a constant
\begin{align}\label{eq:4.12}
q:=\frac{(1+\frac{\theta}{p-1})p}{1+\frac{\theta}{p-1}+p}\in(1,p)
\end{align}
such that
\begin{align}\label{eq:4.13}
\frac{pq}{p-q}=1+\frac{\theta}{p-1}.
\end{align}
Furthermore, by virtue of the last equality and \eqref{eq:2.1} along with the fact that $|b_s|^2\leq \nu_s^2$, according to H\"{o}lder's inequality, we deduce that for each $n\geq1$ and $t\in[0,\tau]$,
\begin{align*}
\begin{split}
&\E\left[(\Delta_t^n)^q\right]\leq \E\left[\left(e^{\int_{0}^{\tau_n}(\mu_s+\frac{\theta}{2(p-1)}|b_s|^2){\rm d}s}U^+_{\tau_n}~e^{\int_{t\wedge\tau_n}^{\tau_n}b_s{\rm d}B_s-\left(\frac{1}{2}+\frac{\theta}{2(p-1)}\right)\int_{t\wedge\tau_n}^{\tau_n}|b_s|^2{\rm d}s}\right)^q\right]\\
& \ \ \leq \left(\E\left[\left(e^{\int_{0}^{\tau_n}\rho_s{\rm d}s}U^+_{\tau_n}\right)^p\right]\right)^{\frac{q}{p}} \left(\E\left[e^{\frac{pq}{p-q}\int_{t\wedge\tau_n}^{\tau_n}b_s{\rm d}B_s-\frac{pq}{p-q}\left(\frac{1}{2}+\frac{\theta}{2(p-1)}\right)\int_{t\wedge\tau_n}^{\tau_n}|b_s|^2{\rm d}s}\right]\right)^{\frac{p-q}{p}}\\
& \ \ \leq \left(\E\left[\sup_{r\in[0,\tau]}\left(e^{\int_{0}^{r}\rho_s{\rm d}s}U^+_{r}\right)^p\right]\right)^{\frac{q}{p}} \left(\E\left[e^{\frac{pq}{p-q}\int_{t\wedge\tau_n}^{\tau_n}b_s{\rm d}B_s-\frac{p^2q^2}{2(p-q)^2}\int_{t\wedge\tau_n}^{\tau_n}|b_s|^2{\rm d}s}\right]\right)^{\frac{p-q}{p}},
\end{split}
\end{align*}
and then it follows from $U_\cdot^+\in S_\tau^p(\rho_\cdot;\R)$ that
\begin{align*}
\sup_{n\geq1}\E\left[(\Delta_t^n)^q\right]<+\infty, \ \ t\in[0,\tau].
\end{align*}
Finally, in light of $\lim\limits_{n\rightarrow\infty}U^+_{\tau_n}=U^+_{\tau}=0$, by letting $n\rightarrow \infty$ in \eqref{eq:4.4} we obtain that $\lim\limits_{n\rightarrow\infty}U^+_{t\wedge\tau_n}e^{\int_{0}^{t\wedge\tau_n}\mu_s{\rm d}s}=0$ for each $t\in[0,\tau]$, that is to say, $Y_t\leq Y'_t, \ t\in[0,\tau]$. The proof is then completed.
\end{proof}

\begin{rmk}\label{rmk:4.3}
In the above process of proving that $(\Delta_t^n)^\infty_{n=1}$ is uniformly integrable for each $t\in[0,\tau]$, we actually expanded $\Delta_t^n$ into the product of $\chi_t^n$ and $\phi_t^n$, where
\begin{align}\label{eq:4.10}
\chi_t^n:=e^{\int_{0}^{\tau_n}(\mu_s+\frac{\theta}{2(p-1)}|b_s|^2){\rm d}s}U^+_{\tau_n}
\end{align}
and
\begin{align}\label{eq:4.11}
\phi_t^n:=e^{\int_{t\wedge\tau_n}^{\tau_n}b_s{\rm d}B_s-\left(\frac{1}{2}+\frac{\theta}{2(p-1)}\right)\int_{t\wedge\tau_n}^{\tau_n}|b_s|^2{\rm d}s},\ \ t\in[0,\tau],
\end{align}
and then the constant $q$ defined in \eqref{eq:4.12} is naturally selected by \eqref{eq:4.13}.
We should particularly mention that \eqref{eq:4.10}-\eqref{eq:4.11} and \eqref{eq:4.12} are different from those used in \cite[Theorem 5.1]{LiFan2024SD} since $\mu_\cdot$ takes values in $\R$ and the scope of $\rho_\cdot$ is extended to \eqref{eq:2.1}.
\end{rmk}

Next, we prove \cref{thm:4.2} by using \cref{thm:3.2} and \cref{thm:4.1}.

\begin{proof}[\bf Proof of \cref{thm:4.2}]
The uniqueness part is a direct consequence of \cref{thm:4.1}. Then, it
suffices to prove the existence part. It is not difficult to verify that \ref{A:H2} implies \ref{A:H1}(i)-(ii) and \eqref{eq:2.3*}, so the existence result follows directly from \cref{thm:3.2}. In fact, we know by \ref{A:H2}(ii) and \ref{A:H2}(v) that \ref{A:H1}(i) holds true. Furthermore, it follows from \ref{A:H2}(iii)-(v) that for each $(t,y,z) \in [0,\tau]\times\R\times\R^d$,
\begin{align}\label{eq:3.1}
\begin{split}
\sgn(y)g(\omega,t,y,z) &\leq |g(\omega,t,0,z)|+\mu_t(\omega)|y|\\
& \leq |g(\omega,t,0,0)| +\mu_t(\omega)|y| + \nu_t(\omega)|z|
\end{split}
\end{align}
and
\begin{align}\label{eq:3.2}
\begin{split}
|g(\omega,t,y,z)| &= |g(\omega,t,y,z)-g(\omega,t,y,0)+g(\omega,t,y,0)-g(\omega,t,0,0)+g(\omega,t,0,0)|\\
&\leq |g(\omega,t,0,0)|+\psi_t(|y|) + \nu_t(\omega)|z|.
\end{split}
\end{align}
Therefore, combining \eqref{eq:3.1} and \eqref{eq:3.2} with assumption \ref{A:H2}(i), we obtain that assumption \ref{A:H1}(ii) and \eqref{eq:2.3*} hold with $|g(\omega,t,0,0)|$ instead of $f_t(\omega)$ and $\phi_t(\omega,x):=|g(\omega,t,0,0)|+\psi_t(x)\in\mathbf{S}$.
\end{proof}

\begin{rmk}\label{rmk:4.2}
We note that for the case that \eqref{eq:2.3} is in force, the uniqueness part of \cref{thm:4.2} has been proved in \citet[Theorem 3.2]{Zhang2026} via \cref{pro:2.2,pro:2.3}, but the existence part of \cref{thm:4.2} cannot be proved via the method used in \citet[Theorem 3.2]{Zhang2026} since \ref{A:H2}(iii) is weaker than \eqref{eq:4.6} used in \citet{Zhang2026}. Furthermore, for the case that $\rho_\cdot$ only satisfies \eqref{eq:2.1} but not \eqref{eq:2.3}, the uniqueness part of \cref{thm:4.2} can be still proved via \cref{pro:2.2}, but the existence part of \cref{thm:4.2} cannot be proved via the method used in \citet{Zhang2026} since \ref{A:H2}(iii) is weaker than \eqref{eq:4.6}, and in light of \cref{rmk:2.1}, \cref{pro:2.3} is no longer true.
\end{rmk}

\section{An existence and uniqueness result of the minimal (maximal) solution}
\label{sec:5-An existence and uniqueness result of the minimal (maximal) solution}
\setcounter{equation}{0}

In this section, using the same weighting factor $e^{\int_0^t \rho_r {\rm d}r}$ as in Sections 3 and 4, we will proceed to establish an existence and uniqueness result of the minimal (maximal) weighted $L^p$ solution for BSDE \eqref{BSDE1.1} and the corresponding comparison theorem, where the generator $g$ satisfies either the stochastic monotonicity condition with general growth in the state variable $y$ and the stochastic linear growth condition in the state variable $z$ or assumption \ref{A:H1'} stated in Section 3.

\subsection{Statement of the main results}

Let us start with introducing the following assumptions on the generator $g$.

\begin{enumerate}
{\addtolength{\leftskip}{2em}\item[\textbf{(H3)} \ (i)]\makeatletter\def\@currentlabel{(H3)}\makeatother\label{A:H3} For each $t\in [0,\tau]$, $g(\omega,t,\cdot,z)$ is continuous uniformly in $z$ and for each $(t,y)\in [0,\tau]\times \R$, $g(\omega,t,y,\cdot)$ is continuous, i.e., for each $(t,y,y_0,z,z_0)\in [0,\tau]\times\R\times\R\times\R^d\times\R^d$, we have
    $$\lim\limits_{y\rightarrow y_0}\sup_{z\in\R^d}|g(\omega,t,y,z)-g(\omega,t,y_0,z)|=0$$
    and $$\lim\limits_{z\rightarrow z_0}|g(\omega,t,y,z)-g(\omega,t,y,z_0)|=0.$$

\item[(ii)] $g$ satisfies a stochastic monotonicity condition in $y$, i.e., for each $(t,y_1,y_2,z)\in [0,\tau]\times\R\times\R\times\R^d$,
$$
\sgn(y_1-y_2)(g(\omega,t,y_1,z)-g(\omega,t,y_2,z))\leq \mu_t(\omega)|y_1-y_2|.
$$

\item[(iii)] $g$ has a general growth in $y$ and a stochastic linear growth in $z$, i.e., there exists a $\phi_t(\omega,x)\in\mathbf{S}$ satisfying $\phi_t(\omega,0)\equiv0$ such that for each $(t,y,z)\in [0,\tau]\times\R\times\R^d$, we have
  $$
  |g(\omega,t,y,z)|\leq f_t(\omega)+\phi_t(\omega,|y|)+\nu_t(\omega)|z|.
  $$}
\end{enumerate}

\begin{rmk}\label{rmk:5.1}
We have the following two remarks.
\begin{itemize}
\item [(i)] If $g$ satisfies \ref{A:H3}(ii)-(iii), then for each $(t,y,z)\in [0,\tau]\times\R \times\R^d$, we have
\begin{align}\label{eq:5.4}
\begin{split}
    \sgn(y)g(\omega,t,y,z)&\leq \sgn(y)(g(\omega,t,y,z)-g(\omega,t,0,z))+|g(\omega,t,0,z)|\\
    &\leq f_t(\omega)+\mu_t(\omega)|y|+\nu_t(\omega)|z|,
\end{split}
\end{align}
which means that $g$ satisfies \ref{A:H1}(ii). Furthermore, it follows from \ref{A:H3}(iii) that $g$ satisfies \eqref{eq:2.3*} with $\bar{\phi}_t(\omega,x):=f_t(\omega)+\phi_t(\omega,x)$ instead of $\phi_t(\omega,x)$. Thus, assumption \ref{A:H3} can imply \ref{A:H1}(i)-(ii) and \eqref{eq:2.3*}. In addition, we would like to mention that \ref{A:H3}(i) is strictly stronger than \ref{A:H1}(i), as can be seen from \cref{ex:3.1} in Section 3, since these generators satisfy \ref{A:H1}(i) but not \ref{A:H3}(i).

\item [(ii)] It is clear that \ref{A:H3}(ii) is identical to \ref{A:H2}(iv). And, it is easy to verify that \ref{A:H3}(iii) can imply \ref{A:H2}(iii). In fact, if $g$ satisfies \ref{A:H3}(iii), by virtue of $\phi_t(0)\equiv0$, then for each $x\in\R_+$, we have
\begin{align*}
\begin{split}
\psi_t(x):=\sup_{|y|\leq x} \left|g(t,y,0)-g(t,0,0)\right|
&\leq \sup_{|y|\leq x}|g(t,y,0)|+|g(t,0,0)|\\
&\leq 2f_t+ \phi_t(x).
\end{split}
\end{align*}
It then follows from \eqref{eq:2.2} and $\int_{0}^{\tau}|\rho_t|{\rm d}t<+\infty$ along with $\phi_t(x) \in \mathbf{S}$ that
\begin{align*}
\begin{split}
\int_0^\tau \psi_t(x){\rm d}t
&\leq 2\int_0^\tau e^{-\int_0^t \rho_r {\rm d}r}e^{\int_0^t \rho_r {\rm d}r}f_t{\rm d}t+\int_0^\tau \phi_t(x){\rm d}t\\
&\leq 2e^{\int_0^\tau |\rho_r| {\rm d}r}\int_0^\tau e^{\int_0^t \rho_r {\rm d}r}f_t{\rm d}t+\int_0^\tau \phi_t(x){\rm d}t<+\infty,
\end{split}
\end{align*}
which means that \ref{A:H2}(iii) holds for $g$, and then the desired assertion is true.\vspace{0.2cm}
\end{itemize}
\end{rmk}

The following \cref{thm:5.1} is the first main result of this section.

\begin{thm}\label{thm:5.1}
Let the generator $g$ satisfy assumption \ref{A:H3} or \ref{A:H1'}. Then, for each $\xi\in L_\tau^p(\rho_\cdot;\R)$, BSDE \eqref{BSDE1.1} admits a minimal (resp. maximal) solution $(y_t,z_t)_{t\in[0,\tau]}$ such that $y_\cdot\in S_\tau^p(\rho_\cdot;\R)$, which means that if $(\bar{y}_t,\bar{z}_t)_{t\in[0,\tau]}$ is any solution of BSDE \eqref{BSDE1.1} such that $\bar{y}_\cdot\in S_\tau^p(\rho_\cdot;\R)$, then for each $t\in[0,\tau]$, $y_t\leq \bar{y}_t$ (resp. $y_t\geq \bar{y}_t$). Furthermore, if \eqref{eq:2.3} holds, then $z_\cdot\in M_\tau^p(\rho_\cdot;\R^d)$.\vspace{0.1cm}
\end{thm}

\begin{rmk}\label{rmk:5.2}
Regarding \cref{thm:5.1}, we make the following remarks.
\begin{itemize}
\item [(i)] Compared to \cref{thm:4.2}, in \cref{thm:5.1} we relax \ref{A:H2}(v) to \ref{A:H3}(iii), while strengthening \ref{A:H2}(ii) to \ref{A:H3}(i).

\item [(ii)] Based on \cref{thm:5.1}, by a similar analysis to \cref{cor:3.1} we can obtain an existence and uniqueness result on a usual minimal (resp. maximal) $L^p$ solution of BSDE \eqref{BSDE1.1} with unbounded stochastic coefficients under assumption \ref{A:H3} or \ref{A:H1'} provided that \eqref{eq:2.4} and \eqref{eq:2.4*} are in force. We note that this result generalizes \citet[Theorem 4.1]{Briand2007}, where the terminal time $\tau$ is a finite positive constant, and both $\mu_\cdot$ and $\nu_\cdot$ in \ref{A:H3}(ii)-(iii) are two nonnegative constants.

\item [(iii)] It is obvious that \ref{A:H1'}(ii) extends both the stochastic linear growth condition (see (H5) therein) and the time-varying linear growth condition (see (H1) therein) used in \citet{Liu2020} and \citet{Fan2011}. Therefore, \cref{thm:5.1} improves those corresponding results obtained in these two papers. In addition, it should be noted that \cref{thm:5.1} also strengthens \cref{cor:3.2}.\vspace{0.1cm}
\end{itemize}
\end{rmk}

\begin{ex}\label{ex:5.1}
We introduce three concrete examples to which \cref{thm:5.1} can be applied, but any existing relevant results, such as \citet[Theorem 4.1]{Briand2007}, \citet[Theorem 1]{Fan2011}, and \citet[Theorem 5.1]{Liu2020} and so on, can not be applied.

\begin{itemize}
\item [(i)] Let $\tau$ be a bounded stopping time, i.e., $\tau\leq T$ for some constant $T>0$ and for each $(t,y,z)\in [0,\tau]\times\R\times\R^{d}$, let
$$
g(t,y,z):=e^{-|B_t|y}-|B_t|y+\sqrt{|B_t|+1}(|z|{\bf 1}_{0\leq |z|\leq 1}+\sqrt{|z|}{\bf 1}_{|z|> 1})-1.
$$
It is not hard to verify that this $g$ satisfies assumption \ref{A:H3} with
$$
f_t\equiv0,\ \ \mu_t:=-|B_t|,\ \ \nu_t:=\sqrt{|B_t|+1},\ \ \phi_t(x):=e^{|B_t|x}+|B_t|x-1,
$$
and \eqref{eq:2.4*} is fulfilled for $p=\theta=2$. It then follows from \cref{thm:5.1} and \cref{rmk:5.2}(ii) that for each $\xi\in L_\tau^2(0;\R)$, BSDE \eqref{BSDE1.1} admits a minimal (resp. maximal) solution $(y_t,z_t)_{t\in[0,\tau]}$ in the space of $H_\tau^2(0;\R\times\R^d)$.

\item [(ii)] Let $\tau$ be a finite stopping time, i.e., $\mathbb{P}(\tau<+\infty)=1$, and for each $(t, y,z)\in [0,\tau]\times \R\times\R^{d}$, let
$$
g(t,y,z):=e^{-t}-(y+y^5)+|z|\sin |z|^2.
$$
Obviously, this $g$ satisfies assumption \ref{A:H3} with
$$
p>\frac{3}{2}, \ \ f_t:=e^{-t},\ \ \mu_t\equiv-1,\ \ \nu_t\equiv1,\ \ \phi_t(x):=x+x^5.
$$
In addition, it is obvious that \eqref{eq:2.4*} is fulfilled for $p>3/2$ and $\theta=2[1\wedge(p-1)]>1$. It then follows from \cref{thm:5.1} and \cref{rmk:5.2}(ii) that for each $\xi\in L_\tau^p(0;\R)$, BSDE \eqref{BSDE1.1} admits a minimal (resp. maximal) solution $(y_t,z_t)_{t\in[0,\tau]}$ in the space of $H_\tau^p(0;\R\times\R^d)$.

\item [(iii)] For each $(t, y,z)\in [0,\tau]\times \R\times\R^{d}$, let
$$
g(t,y,z):=|B_t|^4{\bf 1}_{0\leq t\leq 1}y\cos |z|+|B_t|^2{\bf 1}_{0\leq t\leq 1}|z|\sin y.
$$
Obviously, this $g$ satisfies assumption \ref{A:H1'} with
$$
f_t\equiv0,\ \ \mu_t:=|B_t|^4{\bf 1}_{0\leq t\leq 1},\ \ \nu_t:=|B_t|^2{\bf 1}_{0\leq t\leq 1},
$$
and $\rho_t:=\mu_t+\frac{\theta}{2[1\wedge(p-1)]}\nu_t^2$. By \cref{thm:5.1} we know that for each $\xi\in L_\tau^p(\rho_\cdot;\R)$, BSDE \eqref{BSDE1.1} admits a minimal (resp. maximal) solution $(y_t,z_t)_{t\in[0,\tau]}$ in the weighted space of $H_\tau^p(\rho_\cdot;\R\times\R^d)$.
\end{itemize}
\end{ex}

Finally, we present the following comparison theorem on the minimal (resp. maximal) weighted $L^p$ solutions of BSDEs, which is the second main result of this section.

\begin{thm}\label{thm:5.2}
Assume that $\xi, \xi^{\prime}\in L_\tau^p(\rho_\cdot;\R)$, and both generators $g$ and $g^{\prime}$ satisfy either assumption \ref{A:H3} or assumption \ref{A:H1'}. Let $(y_t,z_t)_{t\in [0,\tau]}$ and $(y_t^{\prime},z_t^{\prime})_{t\in [0,\tau]}$ be, respectively, the minimal (resp. maximal) solution of BSDE$(\xi,\tau,g)$ and BSDE$(\xi^{\prime},\tau,g^{\prime})$ such that both $y_\cdot$ and $y^{\prime}_\cdot$ belong to $S_\tau^p(\rho_\cdot;\R)$. If $\xi\leq \xi^{\prime}$ and $g(\omega,t,y,z)\leq g^{\prime}(\omega,t,y,z)$ for each $(t,y,z)\in[0,\tau]\times\R\times\R^d$, then for each $t\in [0,\tau]$, we have $$y_t\leq y^{\prime}_t.$$
\end{thm}

\subsection{Proof of Theorems \ref{thm:5.1} and \ref{thm:5.2}}\vspace{0.1cm}

Before proving Theorems \ref{thm:5.1} and \ref{thm:5.2}, we introduce the following \cref{pro:5.1,pro:5.2}. Their proofs can be found in the Appendix.
\begin{pro}\label{pro:5.1}
Assume that the generator $g$ satisfies assumption \ref{A:H3} (resp. \ref{A:H1'}). For each $n\geq1$, define the function $g_n$ as follows: for $(\omega,t,y,z)\in \Omega\times [0,\tau]\times\R\times\R^d$,
$$
g_n(\omega,t,y,z):=\inf_{\bar{z}\in\R^d}\{g(\omega,t,y,\bar{z})
+\bar{\nu}_t^n(\omega)|z-\bar{z}|\}
$$
$$
({\rm resp.}\ \ g_n(\omega,t,y,z):=\inf_{(\bar y,\bar{z})\in \R\times \R^d}\{g(\omega,t,\bar y,\bar{z})+\bar\mu_t^n(\omega)|y-\bar y|
+\bar{\nu}_t^n(\omega)|z-\bar{z}|\})
$$
with
$$
\bar\mu_t^n:=\mu_t+ne^{-t}\ \ {\rm and}\ \  \bar{\nu}_t^n:=\sqrt{\nu_t^2 +ne^{-t}}.
$$
Then the sequence of function $g_n$ is well defined, and for each $n\geq1$, $g_n(\omega,t,y,z)$ is $(\F_t)$-progressively measurable for each $(y,z)\in\R\times\R^d$, and it satisfies, $\as$,
\begin{itemize}
\item[(i)] For each $z\in\R^d$, $g_n(\omega,t,\cdot,z)$ is continuous and $g_n$ satisfies \ref{A:H3}(ii)-(iii) (resp. \ref{A:H1'}(ii));
\item[(ii)] Monotonicity in $n$: for each $(y,z)\in\R\times\R^d$, $g_n(\omega,t,y,z)$ nondecreases in $n$;
\item[(iii)] Stochastic Lipschitz condition: for each $(y,y_1,y_2,z_1,z_2)\in\R\times\R\times\R\times\R^d\times\R^d,$ we have
$$
\left|g_n(\omega,t,y,z_1)-g_n(\omega,t,y,z_2)\right|\leq \bar{\nu}_t^n(\omega)|z_1-z_2|
$$
$$
({\rm resp.}\ \ \left|g_n(\omega,t,y_1,z_1)-g_n(\omega,t,y_2,z_2)\right|\leq \bar{\mu}_t^n(\omega)|y_1-y_2|+\bar{\nu}_t^n(\omega)|z_1-z_2|);
$$

\item[(iv)] Convergence: If $(y_n,z_n)\rightarrow (y,z)$, then $g_n(\omega,t,y_n,z_n)\rightarrow g(\omega,t,y,z)$, as $n\rightarrow \infty$.\vspace{0.2cm}
\end{itemize}
\end{pro}

\begin{pro}\label{pro:5.2}
Let $(h_t)_{t\in[0,\tau]}$ be an $(\F_t)$-progressively measurable nonnegative real-valued process satisfying $\int_0^\tau h_t {\rm d}t\leq C_1$ and let the generator $g$ satisfy that
\begin{align}\label{eq:5.21}
|g(t,y,z)|\leq h_t+C_1|z|^2, \ \ \  (t,y,z)\in[0,\tau]\times[-C_2,C_2]\times\R^d,
\end{align}
where $C_i~(i=1,2)$ are two nonnegative constants. Assume that $(y_t,z_t)_{t\in[0,\tau]}$ is an adapted solution of BSDE \eqref{BSDE1.1} such that $|y_\cdot|\leq C_2$. Then $z_\cdot\in M_\tau^2(0;\R^d)$.\vspace{0.2cm}
\end{pro}

Based on the analysis in Section 3, we know that assumption \ref{A:H1'} can imply the assumptions of \cref{thm:3.2}. As seen in \cref{rmk:5.1}(ii), assumption \ref{A:H3} can also imply \ref{A:H1}(i)-(ii) and \eqref{eq:2.3*}. However, it should be especially emphasized that \cref{thm:5.1} can not be proved by the same approach used in the proof of \cref{thm:3.2} since the method of truncation described in \eqref{eq:3.112} is not valid any longer. \vspace{0.2cm}

Next, with the help of \cref{pro:5.1,pro:5.2} along with \cref{thm:4.2,thm:4.1}, we will employ the convolution approaching method and the localization method to prove \cref{thm:5.1}.

\begin{proof}[\bf Proof of \cref{thm:5.1}]
We will prove \cref{thm:5.1} in two steps. Here we only prove the case of the minimal solution. Another case can be proved in the same way.\vspace{0.1cm}

{\bf First Step:} In this step, we prove \cref{thm:5.1} under the assumption \ref{A:H3}. For each $n\geq1$, define
$$
g_n(\omega,t,y,z):=\inf_{\bar{z}\in\R^d}\{g(\omega,t,y,\bar{z})
+\bar{\nu}_t^n(\omega)|z-\bar{z}|\}
$$
with
$$
\bar{\nu}_t^n:=\sqrt{\nu_t^2 +ne^{-t}},\vspace{-0.2cm}
$$
and\vspace{-0.1cm}
\begin{align}\label{eq:5.3}
\rho_t^n:=\rho_t +\frac{\theta}{2(p-1)}ne^{-t}, \ \ t\in[0,\tau].
\end{align}
It is obvious that $\int_0^\tau (|\rho_t^n| +(\bar{\nu}_t^n)^2)<+\infty$ and that the weighted spaces $L_\tau^p(\rho_\cdot^n;\R)$ and $H_\tau^p(\rho_\cdot^n;\R\times\R^d)$ are equivalent to $L_\tau^p(\rho_\cdot;\R)$ and $H_\tau^p(\rho_\cdot;\R\times\R^d)$, respectively. By \eqref{eq:2.1} and \eqref{eq:5.3}, we know that
\begin{align*}
\rho_t^n\geq \mu_t+ \frac{\theta}{2(p-1)}(\bar{\nu}_t^n)^2, \ \ t\in[0,\tau].
\end{align*}

In the sequel, by \cref{thm:4.2} we will first verify that for each $n\geq1$ and $\xi\in L_\tau^p(\rho_\cdot;\R)=L_\tau^p(\rho_\cdot^n;\R)$, the following BSDE$(\xi,\tau,g_n)$:
\begin{align}\label{BSDEg_n}
  y^n_t=\xi+\int_t^\tau g_n(s,y_s^n,z_s^n){\rm d}s-\int_t^\tau z_s^n\cdot{\rm d}B_s,~~t\in[0,\tau].
\end{align}
admits a unique solution $(y_t^n,z_t^n)_{t\in[0,\tau]}$ such that $y_\cdot^n\in S_\tau^p(\rho_\cdot^n;\R)=S_\tau^p(\rho_\cdot;\R)$. In fact, it follows from (i) of \cref{pro:5.1} that for each $n\geq1$, $g_n$ satisfies \ref{A:H3}(iii), and then by \eqref{eq:5.3} and \eqref{eq:2.2}, we have
\begin{align*}
\begin{split}
\E\left[\left(\int_0^\tau e^{\int_0^s \rho_r^n{\rm d}r}|g_n(s,0,0)|{\rm d}s\right)^p\right]
&\leq \E\left[\left(\int_0^\tau e^{\int_0^s \rho_r{\rm d}r}e^{\int_0^s \frac{\theta}{2(p-1)}ne^{-r}{\rm d}r}f_s\right)^p\right]\\
&\leq e^{\frac{n\theta}{2(p-1)}} \E\left[\left(\int_0^\tau e^{\int_0^s \rho_r{\rm d}r}f_s{\rm d}s\right)^p\right]<+\infty,
\end{split}
\end{align*}
which yields that $g_n$ satisfies \ref{A:H2}(i) with $\rho_\cdot^n$ instead of $\rho_\cdot$. Furthermore, by (i) and (iii) of \cref{pro:5.1} along with (ii) of \cref{rmk:5.1} we know that $g_n$ also satisfies \ref{A:H2}(ii)-(v) with $\bar{\nu}_\cdot^n$ instead of $\nu_\cdot$. Therefore, the desired assertion follows immediately from \cref{thm:4.2}.

In light of (ii) of \cref{pro:5.1} and the fact that $\rho_\cdot^{n}\leq \rho_\cdot^{n+1}$, by \cref{thm:4.1} we know that $(y_t^n)_{t\in[0,\tau]}$ is nondecreasing in $n$. Furthermore, since $g_n$ satisfies \ref{A:H3}(ii)-(iii), we can deduce that for each $n\geq1$ and $(t,y,z)\in [0,\tau]\times\R \times\R^d$,
\begin{align}\label{eq:5.5*}
\begin{split}
    \sgn(y)g_n(\omega,t,y,z)&\leq \sgn(y)(g_n(\omega,t,y,z)-g_n(\omega,t,0,z))+|g(\omega,t,0,z)|\\
    &\leq f_t(\omega)+\mu_t(\omega)|y|+\nu_t(\omega)|z|,
\end{split}
\end{align}
which means that $g_n$ satisfies assumption \ref{A:A} with $\bar{\mu}_t:=\mu_t$, $\bar{\nu}_t:=\nu_t$, $\bar{\rho}_t:=\rho_t$ and $\bar{f}_t:=f_t$.
In light of $(y_t^n)_{t\in[0,\tau]}\in S_\tau^p(\rho_\cdot;\R)$, it follows from \cref{pro:2.2} that there exists a constant $C_{p,\theta}>0$ depending only on $p$ and $\theta$ such that for each $n\geq1$ and $t\geq0$,
\begin{align}\label{eq:5.5}
\begin{split}
e^{p\int_{0}^{t}\rho_r{\rm d}r}|y_{t}^{n}|^p &\leq \E\left[\sup_{s\in[t\wedge\tau,\tau]}\left(e^{p\int_{0}^{s}\rho_r{\rm d}r}|y_s^{n}|^{p}\right)\bigg|\F_{t}\right]\\
&\leq
C_{p,\theta}\E\left[e^{p\int_{0}^{\tau}\rho_r{\rm d}r}|\xi|^{p}+\left(
\int_{0}^{\tau}e^{ \int_{0}^{s}\rho_r{\rm d}r}f_s{\rm d}s\right)^{p}\bigg|\F_{t}\right]=:|\tilde{M}_t|^p,
\end{split}
\end{align}
and then
\begin{align}\label{eq:5.6}
|y_{t}^n|\le |\tilde{M}_t| e^{-\int_0^t \rho_r{\rm d}r}=:M_t^{\prime}, \ \ t\in[0,\tau].
\end{align}

Now, for each pair of integers $m,l \geq 1$, we define the following two stopping times:
$$
\bar{\tau}_m = \inf\left\{t\geq0 : M_t^{\prime} \geq m\right\} \wedge \tau \vspace{-0.2cm}
$$
and
$$
\bar{\sigma}_{m,l} = \inf\left\{t\geq0 : \int_0^t (f_s+\phi_s(m)+\nu_s^2){\rm d}s \geq l\right\} \wedge \bar{\tau}_m,\vspace{0.1cm}
$$
with the convention that $\inf \emptyset=+\infty$. Then $(y_{m,l}^{n}(t), z_{m,l}^{n}(t))_{t\in[0,\tau]}:= (y_{t\wedge \bar{\sigma}_{m,l}}^{n}, z_t^{n}1_{t \leq \bar{\sigma}_{m,l}})_{t\in[0,\tau]}$ is a solution to the following BSDE such that $(y_{t\wedge\bar{\sigma}_{m,l}}^{n})_{t\in[0,\tau]}\in S_\tau^p(\rho_\cdot;\R)$:
\begin{align}\label{BSDE5.1}
y_{m,l}^{n}(t) = y_{\bar{\sigma}_{m,l}}^{n} + \int_t^\tau 1_{s \leq \bar{\sigma}_{m,l}} g_{n}(s,y_{m,l}^{n}(s),z_{m,l}^{n}(s)){\rm d}s - \int_t^\tau z_{m,l}^{n}(s)\cdot{\rm d}B_s, \ \ t\in[0,\tau].
\end{align}
Note that for each fixed $m,l\geq1$, $(y_{m,l}^{n}(t))_{t\in[0,\tau]}$ is nondecreasing in $n$ and by (iv) of \cref{pro:5.1} that $\as$, $(g_n)_{n}$ converges locally uniformly in $(y,z)$ to $g$ as $n\rightarrow \infty$. By \eqref{eq:5.6} along with the definitions of $\bar{\tau}_m$ and $\bar{\sigma}_{m,l}$ we can obtain that
\begin{align}\label{eq:5.7}
\as,\ \ \ \sup_{n\geq 1}|y_{m,l}^{n}(t)| \leq m.
\end{align}
Since $g_n$ satisfies \ref{A:H3}(iii) for each $n\geq1$, we know that for each $(t,y,z)\in[0,\tau]\times[-m,m]\times\R^d$,
\begin{align*}
    \sup\limits_{n\geq 1}\left({\bf 1}_{s\leq \bar{\sigma}_{m,l}}|g_n(s,y,z)|\right)
    &\leq {\bf 1}_{s\leq \bar{\sigma}_{m,l}} \left(f_s+\phi_s(m)+\nu_s|z|\right)\\
    &\leq {\bf 1}_{s\leq \bar{\sigma}_{m,l}} \left(f_s+\phi_s(m)+\nu_s^2\right)+|z|^2
\end{align*}
with $\int_0^\tau 1_{s\leq \bar{\sigma}_{m,l}}\left(f_s+\phi_s(m)+\nu_s^2\right) {\rm d}s\leq l$. Furthermore, in light of the last inequality and \eqref{eq:5.7}, it follows from \cref{pro:5.2} that $(z_{m,l}^{n}(t))_{t\in[0,\tau]}\in M_\tau^2(0;\R^d)$.

Next, by an identical argument to the proof of \cref{thm:3.2}, we can also apply Proposition 3.1 in \citet{LuoFan2018} to obtain that for each $m,l\geq 1$, the process
$$
y_{m,l}(t):=\sup_{n\geq1} y_{t\wedge \bar{\sigma}_{m,l}}^{n}
$$
is continuous and the sequence of processes $(z_t^{n}1_{t\leq \bar{\sigma}_{m,l}})_{t\in[0,\tau]}$ converges to $(z_{m,l}(t))_{t\in[0,\tau]}$ strongly in $M_\tau^2(0;\R^d)$ as $n\rightarrow \infty$ such that
$$
y_{m,l}(t)=\sup_{n\geq1} y_{\bar{\sigma}_{m,l}}^{n}+ \int_t^\tau 1_{s\leq \bar{\sigma}_{m,l}} g(s,y_{m,l}(s),z_{m,l}(s)){\rm d}s- \int_t^\tau z_{m,l}(s)\cdot{\rm d}B_s, \ \ t \in [0,\tau].
$$
Moreover, we can conclude that $(y_t,z_t)_{t\in[0,\tau]}$
is an adapted solution to BSDE$(\xi,\tau,g)$, where
\begin{align*}
y_t:=\sup_{n\geq1} y_t^n~~ \text{and}~~ z_t:= \sum_{m=1}^{+\infty}\left( \sum_{l=1}^{+\infty}z_{m,l}(t)1_{t\in[\bar{\sigma}_{m,l-1},\bar{\sigma}_{m,l})}\right)1_{t\in[\bar{\tau}_{m-1},\bar{\tau}_m)}, \ t\in[0,\tau]
\end{align*}
with $\bar{\tau}_0:=1$ and $\bar{\sigma}_{m,0}:=0$ for each $m\geq 1$. And, by sending $n\rightarrow \infty$ in \eqref{eq:5.5} with $t=0$ and using Fatou's lemma, we deduce that $(y_t)_{t\in[0,\tau]}\in S_\tau^p(\rho_\cdot;\R)$.

Furthermore, we verify that $(y_t,z_t)_{t\in[0,\tau]}$ is just the minimal weighted $L^p$ solution of BSDE$(\xi,\tau,g)$. In fact, let $(\bar{y}_t,\bar{z}_t)_{t\in[0,\tau]}$ be any solution of BSDE$(\xi,\tau,g)$ such that $\bar{y}_\cdot\in S_\tau^p(\rho_\cdot;\R)=S_\tau^p(\rho_\cdot^n;\R)$. Note that $(y_t^n,z_t^n)_{t\in[0,\tau]}$ is the unique solution of BSDE \eqref{BSDEg_n} such that $y_\cdot^n\in S_\tau^p(\rho_\cdot^n;\R)$. In light of \cref{pro:5.1}, by \cref{thm:4.1} we obtain that $y_t^n\leq \bar{y}_t$ for each $t\in[0,\tau]$ and $n\geq1$, from which and by letting $n\rightarrow\infty$ we get that for each $t\in[0,\tau]$, $y_t\leq \bar{y}_t$. The desired assertion follows.

Finally, in light of \eqref{eq:5.4}, if \eqref{eq:2.3} holds, then it follows from \cref{pro:2.3} that $z_\cdot\in M_\tau^p(\rho_\cdot;\R^d)$.\vspace{0.2cm}

{\bf Second Step:} In this step, we use the method similar to that in the first step to prove \cref{thm:5.1} under the assumption \ref{A:H1'}.
For each $n\geq1$, define
$$
\bar{g}_n(\omega,t,y,z):=\inf_{(\bar {y},\bar{z})\in \R\times \R^d}\{g(\omega,t,\bar {y},\bar{z})+\bar\mu_t^n(\omega)|y-\bar y|
+\bar{\nu}_t^n(\omega)|z-\bar{z}|\})
$$
with
$$
\bar\mu_t^n:=\mu_t+ne^{-t}\ \ {\rm and}\ \  \bar{\nu}_t^n:=\sqrt{\nu_t^2 +ne^{-t}}
$$
and
\begin{align}\label{eq:5.11}
\bar{\rho}_t^n:=\rho_t +\left(\frac{\theta}{2(p-1)}+1\right)ne^{-t}, \ \ t\in[0,\tau].
\end{align}
It is obvious that $\int_0^\tau (|\bar{\rho}_t^n| +|\bar\mu_t^n|+(\bar{\nu}_t^n)^2)<+\infty$ and that the weighted spaces $L_\tau^p(\bar{\rho}_\cdot^n;\R)$ and $H_\tau^p(\bar{\rho}_\cdot^n;\R\times\R^d)$ are equivalent to $L_\tau^p(\rho_\cdot;\R)$ and $H_\tau^p(\rho_\cdot;\R\times\R^d)$, respectively. By \eqref{eq:2.1} and \eqref{eq:5.11}, we know that
\begin{align*}
\bar\rho_t^n\geq \bar\mu_t^n+ \frac{\theta}{2(p-1)}(\bar{\nu}_t^n)^2, \ \ t\in[0,\tau].
\end{align*}

In the sequel, by \cref{thm:4.2} we will verify that for each $n\geq1$ and $\xi\in L_\tau^p(\rho_\cdot;\R)=L_\tau^p(\bar\rho^n_\cdot;\R)$, the following BSDE$(\xi,\tau,\bar{g}_n)$:
\begin{align}\label{BSDE*g_n}
  y^n_t=\xi+\int_t^\tau \bar{g}_n(s,y_s^n,z_s^n){\rm d}s-\int_t^\tau z_s^n\cdot{\rm d}B_s,~~t\in[0,\tau]
\end{align}
admits a unique solution $(y_t^n,z_t^n)_{t\in[0,\tau]}$ such that $y_\cdot^n\in S_\tau^p(\bar{\rho}_\cdot^n;\R)=S_\tau^p(\rho_\cdot;\R)$. In fact, by (i) of \cref{pro:5.1} we know that for each $n\geq1$, $\bar{g}_n$ satisfies assumptions \ref{A:H2}(ii) and \ref{A:H1'}(ii). Then by \eqref{eq:5.11} and \eqref{eq:2.2} we have
\begin{align*}
\begin{split}
\E\left[\left(\int_0^\tau e^{\int_0^s \bar{\rho}_r^n{\rm d}r}|\bar{g}_n(s,0,0)|{\rm d}s\right)^p\right]
&\leq \E\left[\left(\int_0^\tau e^{\int_0^s \rho_r{\rm d}r}e^{\int_0^s \left(\frac{\theta}{2(p-1)}+1\right)ne^{-r}{\rm d}r}f_s\right)^p\right]\\
&\leq e^{n\left(\frac{\theta}{2(p-1)}+1\right)} \E\left[\left(\int_0^\tau e^{\int_0^s \rho_r{\rm d}r}f_s{\rm d}s\right)^p\right]<+\infty,
\end{split}
\end{align*}
which yields that $\bar{g}_n$ satisfies \ref{A:H2}(i) with $\bar{\rho}_\cdot^n$ instead of $\rho_\cdot$. Moreover, by a similar analysis (ii) of \cref{rmk:5.1}, it follows from \ref{A:H1'}(ii) that $\bar{g}_n$ also satisfies \ref{A:H2}(iii). And, according to (iii) of \cref{pro:5.1}, we know that $\bar{g}_n$ satisfies \ref{A:H2}(iv)-(v) with $\bar{\mu}_\cdot^n$ and $\bar{\nu}_\cdot^n$ instead of $\mu_\cdot$ and $\nu_\cdot$, respectively. Therefore, the desired assertion follows immediately from \cref{thm:4.2}.

In light of (ii) of \cref{pro:5.1} and the fact that $\bar{\rho}_\cdot^{n}\leq \bar{\rho}_\cdot^{n+1}$, by \cref{thm:4.1} we can deduce that $(y_t^n)_{t\in[0,\tau]}$ is nondecreasing in $n$. Furthermore, since $y_\cdot^n\in S_\tau^p(\rho_\cdot;\R)$ and $\bar{g}_n$ satisfies \ref{A:H1'}(ii), it is easy to obtain that \eqref{eq:5.5*}-\eqref{eq:5.6} remain true, and then we have
$$
|y_t^n|\leq M'_t, \ \ t\in[0,\tau].
$$
Finally, for each pair of integers $m,l\geq 1$, define the following two stopping times:
$$ \tilde{\tau}_m = \inf\left\{t\geq0 : M'_t\geq m\right\} \wedge \tau$$
and
$$
\tilde{\sigma}_{m,l} = \inf\left\{t\geq0 : \int_0^t (f_s+m\mu_s+\nu_s^2){\rm d}s \geq l\right\} \wedge \tilde{\tau}_m.\vspace{0.1cm}
$$
The rest of the proof runs as that in the first step. The proof of \cref{thm:5.1} is then complete.
\end{proof}

\begin{rmk}\label{rmk:5.3}
During the above proof, by \cref{pro:5.2} we verify that $(z_{m,l}^{n}(t))_{t\in[0,\tau]}\in M_\tau^2(0;\R^d)$, which enables us to apply \cite[Proposition 3.1]{LuoFan2018} and the localization method.
\end{rmk}

Finally, we give the proof of \cref{thm:5.2}. Here, we will only prove the comparison theorem for the minimal solution. The case of the maximal solution can be proved by a similar argument.

\begin{proof}[\bf Proof of \cref{thm:5.2}]
Assume first that $\xi, \xi^{\prime}\in L_\tau^p(\rho_\cdot;\R)$ satisfy $\xi\leq \xi^{\prime}$, and that generators $g$ and $g^{\prime}$ satisfy \ref{A:H3} and for each $(t,y,z)\in[0,\tau]\times\R\times\R^d$,
\begin{align}\label{eq:5.8}
g(t,y,z)\leq g^{\prime}(t,y,z).
\end{align}
Let $(y_\cdot,z_\cdot)$ and $(y_\cdot^{\prime},z_\cdot^{\prime})$ be, respectively, the minimal solution of BSDE$(\xi,\tau,g)$ and BSDE$(\xi^{\prime},\tau,g^{\prime})$ such that both $y_\cdot$ and $y^{\prime}_\cdot$ belong to $S_\tau^p(\rho_\cdot;\R)$. Let $g_n$, $\bar{\nu}_\cdot^n$ and $\rho_\cdot^n$ be defined as in the proof of \cref{thm:5.1}. Then, for each $n\geq1$, $g_n$ satisfies \ref{A:H2} with $\rho_\cdot^n$ and $\bar{\nu}_\cdot^n$ instead of $\rho_\cdot$ and $\nu_\cdot$, respectively, and BSDE$(\xi,\tau,g_n)$ admits a unique solution $(y_t^n,z_t^n)_{t\in[0,\tau]}$ such that $y_\cdot^n\in S_\tau^p(\rho_\cdot^n;\R)$, and for each $t\in[0,\tau]$,
\begin{align}\label{eq:5.9}
y_t^n \uparrow y_t.
\end{align}
Moreover, it follows from the definition of $g_n$ and \eqref{eq:5.8} that for each $n\geq1$ and $(t,y,z)\in[0,\tau]\times\R\times\R^d$,
$$
g_n(t,y,z) \leq g(t,y,z)\leq g^{\prime}(t,y,z),
$$
which implies that for each $n\geq1$ and $t\in[0,\tau]$,
$
g_n(t,y^{\prime}_t,z^{\prime}_t) \leq g^{\prime}(t,y^{\prime}_t,z^{\prime}_t).
$
Thus, since $g_n$ satisfies \ref{A:H2}(iv)-(v), both $y_\cdot^n$ and $y^{\prime}_\cdot$ belong to $S_\tau^p(\rho_\cdot^n;\R)$, by \cref{thm:4.1} we deduce that for each $n\geq1$,
\begin{align}\label{eq:5.10}
y_t^n \leq y^{\prime}_t,\ \ t\in[0,\tau].
\end{align}
In light of \eqref{eq:5.9}, the desired assertion of $y_\cdot\leq y'_\cdot$ follows by letting $n\rightarrow \infty$ in \eqref{eq:5.10}. In the same way as the above proof process, we can easily obtain the previous assertion under assumption \ref{A:H1'}.
\end{proof}

\section*{Appendix}

\begin{proof}[{\bf Proof of \cref{pro:5.1}.}]
The proof can be regarded as a modification of \citet[Lemma 1]{Lepeltier1997}. For readers' convenience, we list it as follows. Here, we only prove the case that the generator $g$ satisfies assumption \ref{A:H3}, and the other case can be easily proved by an identical argument. We first prove (i). Since $g$ satisfies \ref{A:H3}(iii), we have that for each $n\geq1$ and $(t,y,z) \in[0,\tau]\times\R\times\R^d$,
\begin{align*}
g_n(\omega,t,y,z)&=\inf\limits_{\bar z\in\R^d}\{g(\omega,t,y,\bar z)+\bar\nu_t^n(\omega)|z-\bar z|\}\\
&\geq \inf\limits_{\bar z\in\R^d}\{-f_t(\omega)-\phi_t(\omega,|y|)-\nu_t(\omega)|\bar z|+\nu_t(\omega)|z-\bar z|\} \\
&\geq -f_t(\omega)-\phi_t(\omega,|y|)-\nu_t(\omega)|z|.
\end{align*}
And, by the definition of $g_n$ and \ref{A:H3}(iii) we deduce that for each $n\geq1$ and $(t,y,z) \in[0,\tau]\times\R\times\R^d$,
\begin{align}\label{eq:A.1}\tag{A.1}
g_n(\omega,t,y,z)\leq g(\omega,t,y,z)\leq f_t(\omega)+\phi_t(\omega,|y|)+\nu_t(\omega)|z|,
\end{align}
so $g_n$ satisfies \ref{A:H3}(iii). Now, let us recall two basic inequalities which will be used repeatedly:
\begin{align}\label{eq:A.2}\tag{A.2}
\inf\limits_{x\in D}f(x)-\inf\limits_{x\in D}g(x)\leq \sup\limits_{x\in D}(f(x)-g(x))
\end{align}\vspace{-0.2cm}
and\vspace{-0.2cm}
\begin{align}\label{eq:A.3}\tag{A.3}
|\inf\limits_{x\in D}f(x)-\inf\limits_{x\in D}g(x)|\leq \sup\limits_{x\in D}|f(x)-g(x)|.
\end{align}
By the definition of $g_n$ and \eqref{eq:A.3} we get that for each $n\geq1$ and $(t,y,y_0,z) \in[0,\tau]\times\R\times\R\times\R^d$,
\begin{align*}
&|g_n(\omega,t,y,z)-g_n(\omega,t,y_0,z)|\\
& \ \  = \left|\inf\limits_{\bar z\in\R^d}\{g(\omega,t,y,\bar z)+\bar\nu_t^n(\omega)|z-\bar z|\}-\inf\limits_{\bar z\in\R^d}\{g(\omega,t,y_0,\bar z)+\bar\nu_t^n(\omega)|z-\bar z|\}\right|\\
& \ \ \leq \sup\limits_{\bar z\in\R^d}|g(\omega,t,y,\bar z)-g(\omega,t,y_0,\bar z)|.
\end{align*}
Combining the last inequality with \ref{A:H3}(i) we obtain that $\as$, for each $z\in\R^d$, $g_n(\omega,t,\cdot,z)$ is continuous. Furthermore, since $g$ satisfies \ref{A:H3}(ii), it follows from the definition of $g_n$ and \eqref{eq:A.2} that for each $n\geq1$ and $(t,y_1,y_2,z) \in[0,\tau]\times\R\times\R\times\R^d$, assuming without loss of generality that $y_1>y_2$,
\begin{align*}
&\sgn(y_1-y_2)\left(g_n(\omega,t,y_1,z)-g_n(\omega,t,y_2,z)\right)=g_n(\omega,t,y_1,z)-g_n(\omega,t,y_2,z)\\
& \ \ \ =\inf\limits_{\bar z\in\R^d}\{g(\omega,t,y_1,\bar z)+\bar\nu_t^n(\omega)|z-\bar z|\}-\inf\limits_{\bar z\in\R^d}\{g(\omega,t,y_2,\bar z)+\bar\nu_t^n(\omega)|z-\bar z|\}\\
& \ \ \ \leq \sup\limits_{\bar z\in\R^d}\left[g(\omega,t,y_1,\bar z)-g(\omega,t,y_2,\bar z)\right]\leq\mu_t(\omega)|y_1-y_2|.
\end{align*}
If $y_1<y_2$, it is clear that the last inequality remains true, and then \ref{A:H3}(ii) holds for $g_n$. This completes the proof of property (i). Property (ii) is evident from the definition of $g_n$. It can be derived from the definition of $g_n$ and \eqref{eq:A.3} that for each $n\geq1$ and $(t,y,z_1,z_2) \in[0,\tau]\times\R\times\R^d\times\R^d$,
\begin{align*}
&\left|g_n(\omega,t,y,z_1)-g_n(\omega,t,y,z_2)\right|\\
& \ \ =\left|\inf_{\bar{z}\in\R^d}\{g(\omega,t,y,\bar{z})
+\bar{\nu}_t^n(\omega)|z_1-\bar{z}|\}-\inf_{\bar{z}\in\R^d}\{g(\omega,t,y,\bar{z})
+\bar{\nu}_t^n(\omega)|z_2-\bar{z}|\}\right|\\
& \ \ \leq \sup\limits_{\bar z\in\R^d}\{\bar\nu_t^n(\omega)\big||z_1-\bar z|-|z_2-\bar z|\big|\}\leq\bar\nu_t^n(\omega)|z_1-z_2|,
\end{align*}
from which property (iii) follows immediately. Finally, we show property (iv). If $(y_n,z_n)\rightarrow (y,z)$ as $n\rightarrow \infty$, then by \ref{A:H3}(i) we get that
\begin{align}\label{eq:A.4}\tag{A.4}
\lim\limits_{n\rightarrow\infty}g_n(\omega,t,y_n,z_n)\leq \lim\limits_{n\rightarrow\infty}g(\omega,t,y_n,z_n)=g(\omega,t,y,z).
\end{align}
On the other hand, by the definition of $g_n$ and \ref{A:H3}(iii), for each $n\geq1$, there exists a $z_n^*\in\R^d$ such that
\begin{align}\label{eq:A.5}\tag{A.5}
\begin{split}
&g_n(\omega,t,y_n,z_n)\geq g(\omega,t,y_n,z_n^*)+\bar\nu_t^n(\omega)|z_n-z_n^*|-\frac{1}{n}\\
& \ \ \geq -f_t(\omega)-\phi_t(\omega,|y_n|)-\nu_t(\omega)|z_n^*|+\bar\nu_t^n(\omega)|z_n-z_n^*|-\frac{1}{n}\\
& \ \ \geq -f_t(\omega)-\phi_t(\omega,|y_n|)-\nu_t(\omega)|z_n|+(\bar\nu_t^n-\nu_t)(\omega)|z_n-z_n^*|-\frac{1}{n},
\end{split}
\end{align}
which along with \eqref{eq:A.1} yields that
\begin{align}\label{eq:A.6}\tag{A.6}
(\bar\nu_t^n-\nu_t)(\omega)|z_n-z_n^*|\leq 2f_t(\omega)+2\nu_t(\omega)|z_n|+2\phi_t(\omega,|y_n|)+\frac{1}{n}.
\end{align}
Note that when $n\rightarrow \infty$, $(y_n,z_n)\rightarrow (y,z)$ and $\bar\nu_\cdot^n-\nu_\cdot\rightarrow +\infty$. By \eqref{eq:A.6} we can deduce that $z^*_n\rightarrow z$ as $n\rightarrow \infty$. Furthermore, sending $n$ to infinity on both sides of \eqref{eq:A.5} we obtain that
\begin{align*}
    \lim\limits_{n\rightarrow\infty}g_n(\omega,t,y_n,z_n)\geq \lim\limits_{n\rightarrow\infty}g(\omega,t,y_n,z_n^*)=g(\omega,t,y,z).
\end{align*}
Therefore, combining the last inequality with \eqref{eq:A.4} we can conclude that
$\lim\limits_{n\rightarrow\infty}g_n(\omega,t,y_n,z_n)=g(\omega,t,y,z),$
from which property (iv) follows. The proof of \cref{pro:5.1} is then completed.
\end{proof}

\begin{proof}[{\bf Proof of \cref{pro:5.2}.}]
Define $\gamma:=2 C_1$ and the following function from $\R_+$ into itself by
$$
f(x):=\frac{1}{\gamma^2}(e^{\gamma x}-\gamma x-1).
$$
Clearly, $x\mapsto f(|x|)$ is $C^2$. For each integer $n\geq1$, define the following $(\F_t)$-stopping time:
$$\tau_n:=\inf\left\{t\geq0:\int_0^t|z_s|^2{\rm d}s\geq n\right\}\wedge \tau$$
with the convention that $\inf \emptyset=+\infty$. Applying It\^{o}'s formula to $f(|y_t|)$ yields that for each  $n\geq1$,
\begin{align*}
 f(|y_0|) =& f(|y_{\tau_n}|) + \int_0^{\tau_n} \left(f'(|y_s|)\sgn(y_s)g(s, y_s, z_s)-\frac{1}{2} f''(|y_s|)|z_s|^2 \right) {\rm d}s\\
&- \int_0^{\tau_n} f'(|y_s|)\sgn(y_s)z_s \cdot{\rm d}B_s.
\end{align*}
In light of $|y_\cdot|\leq C_2$ and $f'(x), f''(x)\geq0$ for each $x\geq0$, by \eqref{eq:5.21} and the definition of $\gamma$ we can deduce that there exists a constant $C_3>0$ such that for each $n\geq1$,
\begin{align*}
0\leq f(|y_0|)\leq &C_3\left( 1+\int_0^{\tau_n} h_s{\rm d}s \right)-\int_0^{\tau_n} f'(|y_s|)\sgn(y_s)z_s\cdot {\rm d} B_s\\
&-\frac{1}{2}\int_0^{\tau_n}\left[ (f''(|y_s|)-\gamma f'(|y_s|))|z_s|^2\right]{\rm d}s.
\end{align*}
Finally, by the definition of the stopping time $\tau_n$ and the fact that $f''(x)-\gamma f'(x)=1$ for each $x\geq 0$, taking the mathematical expectation on both sides of the last inequality yields that for each $n\geq1$,
$$\E\left[\int_0^{\tau_n} |z_s|^2 {\rm d}s \right] \leq 2C_3\left( 1 + \int_0^{\tau} h_s{\rm d}s\right)\leq 2C_3(1+C_1).$$
The desired conclusion follows by sending $n\To\infty$ in the last inequality and applying Fatou's lemma.
\end{proof}

\setlength{\bibsep}{2pt}

\end{document}